%% file: ortho-symplectic-draft-2.tex
\documentclass[11pt]{article}

\input{preamble.tex}

\usepackage{fullpage}
\usepackage[nomarkers,figuresonly,nofiglist]{endfloat}

\usepackage[backend=bibtex,doi=false,isbn=false,eprint=false]{biblatex}
\begin{document}
\title{Orthogonal and symplectic orbits in the affine flag variety of type A}

\author{
Kam Hung TONG\thanks{
Department of Mathematics, Hong Kong University of Science and Technology, {\tt khtongad@connect.ust.hk}.
}
}
\date{}

\maketitle

\begin{abstract}
It is a classical result that the set $K\backslash G /B$ is finite, where $G$ is a reductive algebraic group over an algebraically closed field with characteristic not equal to two, $B$ is a Borel subgroup of $G$, and $K = G^{\theta}$ is the fixed point subgroup of an involution of $G$. In this paper, we investigate the affine counterpart of the aforementioned set, where $G$ is the general linear group over formal Laurent series, $B$ is an Iwahori subgroup of $G$, and $K$ is either the orthogonal group or the symplectic group over formal Laurent series.
We construct explicit bijections between the double cosets $K \backslash G/B$ and certain twisted affine involutions. This is the first combinatorial description of $K$-orbits in the affine flag variety of type A.
\end{abstract}

\setcounter{tocdepth}{2}
\tableofcontents

\section{Introduction}
\subsection{Classical background}
Let $G$ be a connected reductive algebraic group over the field of complex numbers $\CC$, and let $B \subset G$ be a Borel subgroup of $G$. Then it is a classical result that the set $B \backslash G/B$ is finite, with a distinct set of double coset representatives forming the Weyl group $W$ \cite[Chapter 27]{Bump}. Thus, the \defn{Bruhat decomposition} can be written as 
\[ G = \bigsqcup_{w \in W} BwB. \]

Now let $\theta = \theta^{-1}$ be a holomorphic involution of $G$. Let $K=G^{\theta} = \{g \in G: \theta(g) = g\}$ be the corresponding fixed point subgroup. Then it is again a classical result that the set $K\backslash G/B$ is finite \cite{Matsuki} (see also \cite{Aomoto, Wolf}). The set $K\backslash G/B$ is treated either as a set of $K$-orbits in the \defn{flag variety} $G/B$, or as $B$-orbits in the \defn{symmetric variety} $K\backslash G$ \cite{Richardson90}, or simply as $(B \times K)$-double cosets in $G$.

The classification of $K\backslash G/B$ when $G$ is any classical linear group of any Lie type is well known \cite{MO}. Denote $\ZZ_{>0} = \{1, 2, 3, \dots \}$ and fix $n \in \ZZ_{>0}$. We focus here on the case when $G=\GL_n(\CC)$. Denote $1_n$ to be the $n$-by-$n$ identity matrix. Then there are only three types of involutions $\theta$ up to conjugacy \cite[Chapter 5, \S1.5]{onishchik2012lie}, namely $\theta(g) = (g^T)^{-1}$, $\theta(g) = (-Jg^TJ)^{-1}$ with $n$ is even and $J = \begin{psmallmatrix} 0 & 1_{n/2} \\ -1_{n/2} & 0\end{psmallmatrix}$, and 
\[\theta(g) = \begin{psmallmatrix} 1_p & 0 \\ 0 & -1_q \end{psmallmatrix} g \begin{psmallmatrix} 1_p & 0 \\ 0 & -1_q \end{psmallmatrix}\] 
with $n = p+q$ for nonnegative integers $p$ and $q$. The corresponding fixed point subgroups $K = G^\theta$ are the orthogonal group $\O_n(\CC)$, the symplectic group $\Sp_n(\CC)$, and the product group 
\[\GL_p(\CC) \times \GL_q(\CC) = \left\{\begin{psmallmatrix} k_{1} & 0 \\ 0 & k_{2} \end{psmallmatrix} : k_{1} \in \GL_p(\CC), k_{2} \in \GL_q(\CC)\right\}\] respectively.
The $K$-orbits of the flag variety $G/B$ in these three cases are in bijection with sets of (signed) involutions in the symmetric group $S_n$ \cite{MO}.
More precisely, the $K$-orbits are in bijection with involutions in $S_n$ for $K = \O_n(\CC)$, fixed-point-free involutions in $S_n$ for $K = \Sp_n(\CC)$ \cite{Richardson90}, and certain signed involutions called \defn{$(p,q)$-clans} for $K = \GL_p(\CC) \times \GL_q(\CC)$ \cite{Yamamoto}.

The earliest record of studying $K$-orbits in the flag variety is probably due to Gelfand and Graev \cite{GelfandGraev}. They showed that one can construct new irreducible (unitary) representations or discrete series representations of non-compact real forms of $G$ from spaces of functions on these $K$-orbits \cite{Graev}.

In the 1980s, Matsuki and Ōshima \cite{MO} gave a concrete combinatorial description of $K\backslash G/B$ for complex classical Lie groups $G$ in terms of \defn{clans}, but without giving a detailed proof. Richardson and Springer \cite{Richardson90} and Yamamoto \cite{Yamamoto} later gave detailed proofs for the three cases of $K$ with $G = \GL_n(\CC)$. Yamamoto \cite{Yamamoto} also considered some of the cases in type B, C, D. A useful overview of the $K$-orbit classifications in these cases are provided in the Ph.D. thesis of Wyser \cite{Wyser}.

These studies of $K$-orbits lead to applications in the representation theory of real groups and other topics. Richardson and Springer considered the weak order of the closure of such orbits \cite{Richardson90, RS93, RS94}. Fulton related these results to Schubert calculus \cite{Fulton96a, Fulton92, Fulton96b}, and Graham \cite{Graham} and Wyser \cite{Wyser} related Fulton's work to certain torus-equivariant cohomologies of the flag variety. 
In recent years, Can, Joyce, and Wyser \cite{CJW} studied the maximal chains in the weak order poset for the three types of $K$-orbits when $G = \GL_n(\CC)$ and specifically described a formula for Schubert classes. More applications can be found in \cite{colarusso2014gelfand,  mcgovern2009closures, mcgovern2009pattern, wyser2014polynomials, wyser2017polynomials}, among many others. 

\subsection{Affine analogs}
This paper is about affine generalizations of the $K$-orbits above in the following sense.

One can consider affine analogs of the $B$-orbits and $K$-orbits of the flag varieties, and consider their applications in cohomology of the affine flag variety \cite{lam2021back}. 
Let $\mathbb{K}$ be a quadratically closed field, i.e. a field of characteristic not equal to 2 in which every element has a square root.
Let $\KLt$ be the field of formal Laurent series in $t$ consisting of all the formal sums $\sum^{\infty}_{i\geq N} a_i t^i$, in which $N \in \ZZ$ and $a_i \in \K$ for $i \geq N$. Let $\KPt$ be the ring of formal power series consisting of all the formal sums $\sum^{\infty}_{i\geq 0} a_i t^i$, in which $a_i \in \K$ for $i \geq 0$. 

We redefine $G = \GL_n(\KLt)$ to be the group of invertible $n$-by-$n$ matrices over $\KLt$ and redefine $B$ to be the subgroup consisting of all upper triangular modulo $t$ matrices in $\GL_n(\KPt)$, that is, invertible matrices with entries in $\KPt$ that become upper triangular if we set $t = 0$ for these matrices. Then $G$ is the \defn{(algebraic) loop group} of $\GL_n(\K)$ and $B$ is an \defn{Iwahori subgroup}. In this setting, the \defn{affine Bruhat decomposition} is written as 
\[G =\bigsqcup_{w \in \widetilde W} BwB,\]
 where $\widetilde{W}$ is the \defn{affine Weyl group} of $G$, which is isomorphic to a semidirect product of the symmetric group $S_n$ of permutations of $n$ elements and $\ZZ^n$ of $n$-tuples of integers. More details on the affine Weyl group of $\GL_n(\KLt)$ are given in Section~\ref{chpt2-ref}. The set of cosets $G/B$ is often called the \defn{affine flag variety}, and has been studied in \cite{lam2021back, lee2019combinatorial, Nadler04}, for example.
 
Most constructions related to orbits in flag varieties have affine analogs \cite{lam2021back, lee2019combinatorial, Nadler04}. However, the subject of $K$-orbits in the affine flag variety is mostly unexplored in the literature, with the important exception of the Ph.D. thesis of Mann \cite{Mann}. 
Mann's work, which we review in Section~\ref{mann-thesis-ref}, gives a type-independent classification of the $K$-orbits in $G/B$ in terms of certain conjugacy classes of triples $(H, B, \mu)$.
However, this general classification is non-constructive and its computation is complicated even for matrix groups with small dimensions. By contrast, the results in this paper provide explicit combinatorial descriptions of $K$-orbits in the affine flag variety of type A. It is not straightforward to obtain our results as special cases of Mann's theorem.
The results of this paper concern the $K$-orbits in the affine flag variety $G/B$, where $G = \GL_n(\KLt)$ is the (algebraic) loop group of $\GL_n(\K)$ and $B$ is the Iwahori subgroup as in the setting of affine Bruhat decomposition. In type A, the loop group analogs of $K$ are given by $\O_n(\KLt)$ and $\Sp_n(\KLt)$. We also consider the $\SO_n(\KLt)$-orbits in the affine flag variety of $\SL_n(\KLt)$. The following three subsections describe our main theorems in these three cases. A separate paper will treat the case where $K = \GL_p(\KLt) \times \GL_q(\KLt)$.

\begin{remark}
In this paper, by orbits in affine flag variety we mean $K$-orbits in $G/B$. However, it is sometimes more convenient to consider the $B$-orbits in $K \backslash G$ or the $(K,B)$-double cosets in $G$, which are the orbits for the obvious action of $K \times B$ on $G$. It is clear that there are canonical bijections between these three kinds of orbits, but most of our results are phrased in terms of $(K,B)$-double cosets.
\end{remark}

\subsection{Orbits of the orthogonal group}
Continue to let $G = \GL_n(\KLt)$. 
In this subsection we state our first main theorem classifying the orbits of 
\[K = \O_n(\KLt) = \{g \in G : g^Tg = 1_n\}\]
 in the affine flag variety $G/B$. Here the involution $\theta$ on $G$ is defined as $\theta(g) = (g^T)^{-1}$ and $K = G^{\theta}$. 

Recall that a \defn{monomial matrix} is a matrix with only one non-zero entry in each row and column.
We define an \defn{affine permutation matrix} to be an $n$-by-$n$ monomial matrix with integral powers of $t$ as non-zero entries. 
Below we define two subsets of the set of affine permutation matrices.

\begin{definition}\label{def-extended-aff-twisted-inv-1}
Define $\affinvol_n$ to be the set of all symmetric $n$-by-$n$ affine permutation matrices. Define $\eaffinvol_n$ to be the set of elements in $\affinvol_n$ for which  the sum of the powers of $t$ is even.
\end{definition}
The following are three affine permutation matrices:
\[ \begin{pmatrix}
0 & 0 & t^2 & 0 & 0 & 0\\
t^{1} & 0 & 0 & 0 & 0 & 0\\
0 & t^{-1} & 0 & 0 & 0 & 0\\
0 & 0 & 0 & 0 & 0 & t^3\\
0 & 0 & 0 & 0 & t^{-2} & 0\\
0 & 0 & 0 & t^0 & 0 & 0
\end{pmatrix}, \,
\begin{pmatrix}
t& 0 & 0 & 0 & 0\\
0 & 0 & 0 & t^{-2} & 0\\
0 & 0 & t^{-5} & 0 & 0\\
0 & t^{-2} & 0 & 0 & 0\\
0 & 0 & 0 & 0 & t^3
\end{pmatrix}, \,
 \begin{pmatrix}
t^4& 0 & 0 & 0 & 0\\
0 & 0 & 0 & t^{-2} & 0\\
0 & 0 & t^{-5} & 0 & 0\\
0 & t^{-2} & 0 & 0 & 0\\
0 & 0 & 0 & 0 & t^3
\end{pmatrix}.
\]
Only the second and the third matrices are in $\affinvol_5$, while only the third matrix is in $\eaffinvol_5$.

The main theorem for the case $K = \O_n(\KLt)$ is the following:

\begin{theorem}\label{ortho-orbit-thm-intro}
In the case where $K = \O_n(\KLt)$ and $G = \GL_n(\KLt)$,
for each double coset $\mathcal{O} \in  K\backslash G/B$, there exists a unique $w \in \eaffinvol_n$ such that $g^Tg = w$ for some $g \in \mathcal{O}$.
Moreover, for each $w \in \eaffinvol_n$, the set of matrices $g$ satisfying $g^Tg = w$ is non-empty and its elements lie in the same double coset.
\end{theorem}
This theorem is non-constructive but
for each $w \in \eaffinvol_n$, we provide an explicit formula for a matrix $g_w \in G$ such that $g_w^Tg_w = w$. See Definition~\ref{def-gw-ortho} and \eqref{sqrt-ortho-affineperm} for more details.

In terms of this notation, the above theorem implies the following corollary.
\begin{corollary}
The map $w \mapsto Kg_wB$ is a bijection between $\eaffinvol_n$ and $K\backslash G/B$. 
\end{corollary}

The proof of Theorem~\ref{ortho-orbit-thm-intro} appears in Section~\ref{sect3.1-ref}.

\subsection{Orbits of the special orthogonal group}
For any commutative ring $R$, denote $\SL_n(R)$ to be the special linear group over $R$. In this subsection we let $G$ be the group $\SL_n(\KLt)$ and $B$ be the Iwahori subgroup consists of upper triangular matrices modulo $t$ in $\SL_n(\KPt)$.

Now we state a theorem classifying the orbits of 
\[K = \SO_n(\KLt) = \{g \in \SL_n(\KLt) : g^Tg = 1_n\}\]
 in the affine flag variety $G/B$. Here $K = G^\theta$ for $\theta(g) = (g^T)^{-1}$.
 
Below, we give the definition of the indexing set for this case.

\begin{definition}\label{def-aff-tw-invol}
Define $i\affinvol_n \subset \SL_n(\KLt)$ to be the set of symmetric monomial matrices such that 
\ben
\item[(i)]if there are any non-zero entries on the diagonal, these diagonal entries are of the form $t^a$, and the non-diagonal non-zero entries are of the form $i = \sqrt{-1}$ times $t^a$, where $a \in \ZZ$.
\item[(ii)]if there are no non-zero entries on the diagonal, then the non-zero entries in the first row and first column are of the form $\pm it^a$, while the remaining non-zero entries are of the form $it^a$, where $a \in \ZZ$.
\een
\end{definition}

The following are three matrices in either $i\affinvol_5$ or $i\affinvol_4$:
\[\begin{pmatrix}
t^6& 0 & 0 & 0 & 0\\
0 & 0 & 0 & it^{-2} & 0\\
0 & 0 & t^{-5} & 0 & 0\\
0 & it^{-2} & 0 & 0 & 0\\
0 & 0 & 0 & 0 & t^3
\end{pmatrix}, \,
 \begin{pmatrix}
0 & 0 & 0 & it^{2} \\
0 & 0 & it^{-2} & 0 \\
0 & it^{-2} & 0 & 0 \\
it^2& 0 & 0 & 0 
\end{pmatrix} , \,
 \begin{pmatrix}
0 & 0 & 0 & -it^{2} \\
0 & 0 & it^{-2} & 0 \\
0 & it^{-2} & 0 & 0 \\
-it^2& 0 & 0 & 0 
\end{pmatrix}.
\]

The following is the main theorem in this subsection:
\begin{theorem}\label{sortho-orbit-thm-intro}
In the case where $K = \SO_n(\KLt)$ and $G = \SL_n(\KLt)$,
for each double coset $\mathcal{O} \in  K\backslash G/B$, there exists a unique $w \in i\affinvol_n$ such that $g^Tg = w$ for some $g \in \mathcal{O}$.
Moreover, for each $w \in i\affinvol_n$, the set of matrices $g$ satisfying $g^Tg = w$ is non-empty and its elements lie in the same double coset.
\end{theorem}

The proof of this theorem appears in Section~\ref{sect3.3-ref}.

For each $w \in i\affinvol_n$, we define explicitly $g_w \in \SL_n(\KLt)$ as a double coset representative satisfying $g_w^Tg_w = w$ by Definition~\ref{def-gw-sortho} and Lemma~\ref{sqrt-sortho-affineperm}. 
The above theorem implies the following corollary.
\begin{corollary}
The map $w \mapsto Kg_wB$ is a bijection between $i\affinvol_n$ and $K\backslash G/B$. 
\end{corollary}

\subsection{Orbits of the symplectic group}
In this subsection let $G = \GL_{2n}(\KLt)$. We state our main theorem classifying the orbits of 
\[K = \Sp_{2n}(\KLt) = \{g \in \GL_{2n}(\KLt) : g^TJg = J\}\]
 in the affine flag variety $G/B$, where $J = \begin{psmallmatrix} 0 & 1_n \\ -1_n & 0\end{psmallmatrix}$. Here $K = G^{\theta}$ for $\theta(g) = (-Jg^TJ)^{-1}$. Below we define another subset of $G$ consisting skew-symmetric matrices.
 
 \begin{definition}\label{def-fpf-ex-aff-tw-invol}
The set $\skewsym_{2n}$ consists of all skew-symmetric $2n$-by-$2n$ monomial matrices whose non-zero entries above the diagonal are integral powers of $t$.
\end{definition}

An example of a matrix in $\skewsym_{2n}$ is 
\[ \begin{pmatrix}
0 & t^2& 0 & 0 \\
- t^2 & 0 & 0 & 0 \\
0 & 0 & 0 & t^{-3} \\
0 & 0 & -t^{-3} & 0
\end{pmatrix}.\] 

The main theorem for the case $K = \Sp_{2n}(\KLt)$ is the following:

\begin{theorem}\label{sp-orbit-thm-intro}
In the case where $K = \Sp_{2n}(\KLt)$ and $G = \GL_{2n}(\KLt)$,
for each double coset $\mathcal{O} \in K\backslash G/B$, there exists a unique $w \in \skewsym_{2n}$ such that $g^{-1}Jg =w$ for some $g \in \mathcal{O}$. 
Moreover, for each $w \in \skewsym_{2n}$, the set of matrices $g$ satisfying $g^{-1}Jg = w$ is non-empty and its elements lie in the double coset. 
\end{theorem}

For each $w \in \skewsym_{2n}$, we give an explicit formula for a matrix $g_w \in \GL_{2n}(\KLt)$ such that $g_w^{-1}Jg_w = w$. See Definition~\ref{def-gw-sp} for the precise definition of $g_w$. 

The above theorem implies the following corollary.
\begin{corollary}
The map $w \mapsto Kg_wB$ is a bijection between $\skewsym_{2n}$ and $K \backslash G/B$.
\end{corollary}

The proof of Theorem~\ref{sp-orbit-thm-intro} appears in Section~\ref{sect4-ref}.
\section*{Acknowledgments}
We thank Eric Marberg, Yongchang Zhu, and Peter Magyar for many useful comments and discussions.
\section{Preliminaries}\label{chpt2-ref}

In this section, we introduce the necessary notations and preliminaries for this paper, focusing on the concept of loop groups, the (extended) affine Weyl group, and the affine flag variety. These topics have various approaches and contain rich theories, but Magyar's work \cite{Magyar} provides a comprehensive exposition on them. For more detailed information, refer to \cite{Kumar, PS}.
\subsection{Loop groups}

Let $F = \KLt = \left\{ \displaystyle\sum^{\infty}_{i\geq N} a_i t^i : N \in \ZZ, a_i \in \K \right\}$ be the field of \defn{formal (complex) Laurent series}. We define the order \defn{$\ord$} of a formal Laurent series $f = \displaystyle\sum^{\infty}_{i = -\infty} a_i t^i \in \KLt$ to be the smallest integer $N$ such that $a_N \neq 0$. 

The ring of \defn{formal (complex) power series}, denoted by $A = \KPt$, consists of all formal Laurent series $f$ with $\ord(f) \geq 0$. The ring $\KPt$ is a discrete valuation ring with a valuation $\ord$, and $\KLt$ is the quotient field of $\KPt$ with the same discrete valuation $\ord$.

It is well-known that the units $\KPt^*$ in $\KPt$ are precisely the elements of $\KPt \backslash t\KPt$, hence $\KPt$ is a local ring with a maximal ideal generated by $t$.

Notice that $\ord(f) \geq 0$ if and only if $f \in \KPt$, and $\ord(f) = 0$ if and only if $f \in \KPt^*$.

Moving forward, we consider a fixed natural number $n$. For any commutative ring $R$, let $\GL_n(R)$ be the group of invertible matrices with entries in $R$. The group $G = \GL_n(\KLt)$ is often called the \defn{(algebraic) loop group} of $\GL_n(\K)$, because $\GL_n(\CCCt)$ can be viewed as a completion of the group of polynomial maps from the circle $S^1 \subset \CC^{\times}$ to $\GL_n(\CC)$ \cite[\S 1.1]{Magyar}. 

We define $B_n \subset \GL_n(\KPt)$ to be the \defn{Iwahori subgroup} consists of upper triangular matrices modulo $t$. For example,
\[ \begin{pmatrix} 1 & 1+t & \sum_{i\geq 0}^{\infty}t^i \\
0 & 1+t^2 & t + t^3\\
0 & 0 & -1 + t^2 + t^4 \end{pmatrix} \in B_3, \quand
\begin{pmatrix} 1 & 1+t & \sum_{i\geq 0}^{\infty}t^i \\
t & 1+t^2 & t + t^3\\
 \sum_{i\geq 1}^{\infty}t^i & t + t^3 + t^5 & -1 + t^2 + t^4 \end{pmatrix} \in B_3, \]
while $\begin{pmatrix} 1 & 1+t & \sum_{i\geq 0}^{\infty}t^i \\
0 & 1 & t + t^3\\
0 & 0 & t \end{pmatrix} \notin B_3$, since its determinant is equal to $t$, which is not a unit in $\KPt$, and hence the matrix is not invertible in $\GL_3(\KPt)$.

If $n$ is clear from context, we write $B$ instead of $B_n$ for the Iwahori subgroup.

\subsection{Extended affine Weyl group}
The extended affine Weyl group and affine Weyl group of $\GL_n(\KLt)$ are concretely described as the following.

Let $\Perm(\ZZ)$ denote the group of bijective maps from $\ZZ$ to $\ZZ$, with composition of maps as the group operation.
Let $\sigma$ be the shift bijection with $\sigma(i) = i+1$ for all $i \in \ZZ$, and let $\tau = \sigma^n$, so that $\tau(i) = i+n$ for all $i \in \ZZ$. 
\begin{definition}
The \defn{extended affine symmetric group} $\widetilde S_n^+$ is the subgroup of $\Perm(\ZZ)$ consisting of bijective maps $w: \ZZ \rightarrow \ZZ$ satisfying $w(i+n) = w(i) + n$ for all $i \in \ZZ$. We refer to elements $w \in \widetilde S_n^+$ as an \defn{affine permutations}.
\end{definition}

Throughout this paper, denote the set $\{1, 2, \dots, n\}$ as $[1,n]$. For $\mathbf{c} = (c_1, c_2, \dots, c_n) \in \ZZ^n$, we define $\tau^{\mathbf{c}}$ by the formula
 \[\tau^{\mathbf{c}}(i + kn) = i + kn + nc_i\]
  for $i \in [1,n]$ and $k \in \ZZ$. This gives an embedding from $\ZZ^n$ to $\Perm(\ZZ)$. 

Define $S_n$ as the subgroup of $\widetilde S_n^+$ consists of affine permutations preserving $[1,n]$. This subgroup $S_n$ is isomorphic to the permutation group of $n$ elements.

In fact we can express any element $\pi \in \widetilde S_n^+$ uniquely as 
\[ \pi = \bar \pi \tau^{\mathbf{c}},\]
where $\bar \pi \in S_n$ and $\mathbf{c} \in \ZZ^n$.
 For any two elements $\pi_1 = \bar \pi_1 \tau^{\mathbf{c_1}}$ and $\pi_2 = \bar \pi_2 \tau^{\mathbf{c_2}}$ in $\widetilde S_n^+$, we have
\[ \pi_1 \pi_2 = \bar \pi_1 \tau^{\mathbf{c_1}}\bar \pi_2 \tau^{\mathbf{c_2}} =\bar \pi_1 \bar \pi_2 \tau^{\bar \pi_2^{-1} (\mathbf{c_1})} \tau^{\mathbf{c_2}} =  \bar \pi_1 \bar \pi_2 \tau^{\bar \pi_2^{-1} (\mathbf{c_1}) + \mathbf{c_2}}.  \]
This shows that $\widetilde S_n^+$ is isomorphic to a semi-direct product $S_n \ltimes \ZZ^n$.
\begin{definition}
The \defn{affine symmetric group} is the subgroup $\widetilde{S_n} = \{ \pi \in \widetilde S_n^+ : \sum^n_{i = 1} (\pi(i) - i) = 0 \}$ of the extended affine symmetric group $\widetilde S_n^+$.
\end{definition} 
We have $\widetilde{S_n} \cong S_n \ltimes \ZZ_0^n$, where $\ZZ_0^n = \{\mathbf{c} \in \ZZ^n : \sum^n_{i=1} c_i =0 \} = \ZZ^n \cap \widetilde{S_n}$. The group $\widetilde{S}_n$ is the affine symmetric group of extended Dynkin type $\widehat A_{n-1}$.

For $i \in \ZZ$, define the \defn{simple transposition} $s_i \in \widetilde S_n^+$ as the affine permutation satisfying $s_i(i + kn) = i+1 + kn$ and $s_i(i+1+kn) = i+kn$ for $k \in \ZZ$, and $s_i(j) = j$ for $j \not\equiv i, i+1$ mod $n$. It follows that $s_i = s_{i+n}$ for any $i \in \ZZ$.

The subgroup $S_n \subset \widetilde S_n^+$ is generated by $s_1, s_2, \dots, s_{n-1}$, while the group $\widetilde{S_n}$ is generated by $s_1, s_2, \dots s_{n-1}$ and $s_0 = s_n$.
\begin{remark}
The groups $\widetilde S_n$ and $\widetilde S_n^+$ are the \defn{affine Weyl group} and the \defn{extended affine Weyl group} of $\GL_n(\KLt)$ respectively.
In the literature, these (extended) affine Weyl groups are also defined in the following way.
Let $T$ be the subgroup of $\GL_n(\KLt)$ consisting of diagonal matrices, and let $\hat T$ be the subgroup of $\GL_n(\KPt)$ consisting of diagonal matrices.

Recall that if $H'$ is a subgroup of $H$, then the \defn{normalizer} of $H'$ in $H$, denoted as $N_H(H')$, is defined as the group $\{h \in H : h^{-1}H'h = H'\}$.
For $G = \GL_n(\KLt)$, the extended affine Weyl group is defined as the quotient $N_G(\hat T)/\hat T$ and the affine Weyl group is defined as the quotient $N_{G_0}(\hat T) / \hat T$, where $G_0 = \{ g\in G : \ord (\det g) = 0\}$. It holds that 
$ \widetilde S_n^+ \cong N_G(T)/T$
 and 
 $\widetilde{S_n} \cong N_{G_0}(\hat T) / \hat T$.
  See \cite[\S 1.2]{Magyar} for more details.
\end{remark}

Below, we describe a way to identify elements in $S_n$ and $\widetilde S_n^+$ with matrices in $\GL_n(\KLt)$. Let $e_1, e_2, \dots, e_n$ denote the standard $\KLt$-basis of $\KLt^n$, and also the standard $\ZZ$-basis of $\ZZ^n$. For a permutation $\pi \in S_n$, the \defn{permutation matrix} associated with $\pi$ is an $n$-by-$n$ matrix (also denoted as $\pi$) such that $\pi e_i = e_{\pi(i)}$ for all $i \in [1,n]$. In other words, the matrix $\pi$ is an $n$-by-$n$ monomial matrix with $1$'s in the entries $(\pi(i), i)$ for all $i \in [1,n]$. Recall that  a \defn{monomial matrix} is a matrix with only one non-zero entry in each row and column, and an \defn{affine permutation matrix} in $\GL_n(\KLt)$ is a monomial matrix with non-zero entries being integral powers of $t$.

Similar to the permutations in $S_n$, for an affine permutation $\pi = \bar \pi \tau^{\mathbf{c}} \in \widetilde S_n^+$, the affine permutation matrix associated with $\pi$ is an $n$-by-$n$ matrix (also denoted as $\pi$) such that $\pi e_i = t^{c_i}e_{\pi(i) \text{ mod } n}$ for all $i \in [1,n]$. In other words, the matrix $\pi$ is an $n$-by-$n$ monomial matrix with $t^{c_i}$'s in the entries $(\bar \pi(i), i)$ for $i \in [1,n]$.

For example if $n = 6$ and $\pi = (1\, 2 \, 3) (4 \, 6) \tau^{(1,-1,2,0, -2, 3)}$, then the affine permutation matrix associated with $\pi$ is 
\[\begin{pmatrix}
0 & 0 & t^2 & 0 & 0 & 0\\
t^{1} & 0 & 0 & 0 & 0 & 0\\
0 & t^{-1} & 0 & 0 & 0 & 0\\
0 & 0 & 0 & 0 & 0 & t^3\\
0 & 0 & 0 & 0 & t^{-2} & 0\\
0 & 0 & 0 & t^0 & 0 & 0
\end{pmatrix} .
\]

\begin{remark}\label{affine-perm-as-matrix}
Using the conventions in the previous paragraphs, we can now identify elements in $\widetilde S_n^+$ as affine permutation matrices, and elements in $\widetilde{S_n}$ as affine permutation matrices for which the sum of the powers of $t$ of all non-zero entries is equal to $0$.\end{remark}

\subsection{Discussion of related work}\label{mann-thesis-ref}

$K$-orbits in the affine flag variety were also studied in the Ph.D. thesis of Mann \cite{Mann}. In this section we explain the relationship between the results in this paper and Mann's work. Mann's thesis addresses a problem mentioned by Lusztig \cite{Lusztig} concerning affine generalizations of the results in his paper with Vogan \cite{lusztig1983singularities}.

Vogan \cite{Vogan} showed that in the finite case the $K$-orbits on the flag variety are in bijection with the $K$-conjugacy classes of the pair $(H,B)$, where $H$ is a $\theta$-stable maximal torus in $G$, and $B$ is a $\theta$-stable Borel subgroup of $G$ containing $H$. For the moment, we denote $G(\CCCt)$ as the loop group of a semisimple complex Lie group $G$. To distinguish the Iwahori subgroup from the Borel subgroup, we denote the Iwahori group by $I$ in this subsection only. 

The following theorem is one of the main results in Mann's thesis.

\begin{theorem}[{\cite[Theorem 6.1]{Mann}}]
The $K(\CCCt)$-orbits on the affine flag variety $G(\CCCt)/I$ are in bijection with the $K$-conjugacy classes of the triples $(H, B, \mu)$, where $H$ is a $\theta$-stable maximal torus in $G$, $B$ is a $\theta$-stable Borel subgroup containing $H$, and $\mu$ is a split cocharacter of $H$, meaning that $\mu \in \Hom(\CC^*, H)$ such that $\theta \mu = \mu^{-1}$.
\end{theorem}

In Vogan's work and Mann's thesis, the indexing set is slightly different. The $B$ in the pairs $(H,B)$ or the triples $(H, B, \mu)$ is replaced by $\Delta^+$, a positive root system of the Lie algebra $\mathfrak{h}$ of torus $H$ instead. However, by the root space decomposition of the Lie algebra $\mathfrak{g}$ of $G$, a positive root system of the Lie algebra $\mathfrak{h}$ of torus $H$ uniquely determines a Borel subalgebra $\mathfrak{b}$ containing $\mathfrak{h}$, which in turn determines a Borel subgroup $B$ by the exponential map. The $\theta$-stable property of Lie groups carries through to the differential $d \theta$-stable property of Lie algebras (see, for example, \cite[Thm. 3.28]{hall2013lie}), and so $B$ is $\theta$-stable. Therefore each triple $(H, \Delta^+ , \mu)$ encodes the same data as $(H, B, \mu)$ for some $\theta$-stable Borel subgroup $B$ containing $H$.

Here we clarify what is meant by $K$-conjugacy of the triples $(H, B, \mu)$. Two triples $(H, B, \mu)$ and $(H', B', \lambda)$ are $K$-conjugate if there is a $k \in K$ such that $H' = \textrm{Ad}_kH$, $B' = \textrm{Ad}_kB$ and $\lambda =\textrm{Ad}_k \circ \mu$, where $\textrm{Ad}_k$ is the map from $G$ to itself given by the formula $\textrm{Ad}_k(g) = k^{-1} gk$.

The main results of this paper also provide such an orbit classification, but only in type A affine flag variety in which $G = \GL_n(\CCCt)$. Mann's classification is more general but much less explicit. As will be discussed below, it is not straightforward to obtain our results by specializing Mann's classification to type A. Moreover, Mann's work has not been published and is not easily accessible.

Here, we quote the explicit example given in Mann's thesis \cite{Mann} to illustrate the triples $(H, B, \mu)$. We consider $G = \SL_2(\CC)$ and $\theta(g) = (g^T)^{-1}$, so that $K = G^\theta = \SO_2(\CC)$. Explicit calculations show that in this case, the distinct double coset representatives of $K(\CCCt) \backslash G(\CCCt)/ I$ are (for more details, see Lemma~\ref{sorthogonaln=2-gl2}):
\[ \begin{pmatrix}t^n & 0 \\ 0 & t^{-n} \end{pmatrix} \quad \text{ with $n \in \ZZ$, } \quad \begin{pmatrix}0 & i \\ i & 0 \end{pmatrix} \quand \begin{pmatrix}0 & -i \\ -i & 0 \end{pmatrix} \text{ with $i = \sqrt{-1}$. } \]
Now, we list the triples $(H, B, \mu)$ for this example. The $K$-conjugacy classes of $\theta$-stable maximal tori consist of $H$, the diagonal matrices, and $H' = K = \SO_2(\CC)$.

For the case when $H$ is the group of diagonal matrices, all $\mu$'s are of the form $\mu: z \mapsto \begin{psmallmatrix} z^n & 0 \\ 0 & z^{-n} \end{psmallmatrix}$, where $n \in \ZZ$. Therefore, all $\mu$'s are split, and $\Hom(\CC^*, H) \cong \ZZ$ as groups. There are two choices of positive roots, namely $\alpha = e_{11}^* - e_{22}^* \in \mathfrak{h}^* = \Hom_{\CC}(\mathfrak{h}, \CC)$, where $\mathfrak{h}$ is the Lie algebra of $H$, and $-\alpha = e_{22}^* - e_{11}^*$. However, in this case, $d\theta \alpha = -\alpha$ (the $\theta$ action on the positive roots is induced from its action on $H$). Therefore, the two positive root systems, corresponding to the two Borel subgroups of upper or lower triangular matrices in $G$, are in the same $K$-conjugacy class. Consequently, the orbit representatives for this $H$ are 
\[(H, B, \mu \in \Hom(\CC^* , H) \cong \ZZ),\]
 where $B$ is the standard Borel subgroup of $G$, and the triples correspond to the double coset representatives $\begin{psmallmatrix}t^n & 0 \\ 0 & t^{-n} \end{psmallmatrix}$.

In the other case, where $H' = K = \SO_2(\CC)$, all cocharacters are fixed by $\theta$, and thus the only split cocharacter is $\mu = 0 \in \ZZ$ or the identity map. Furthermore, the small Weyl group, which counts the number of $K$-conjugate positive root systems, is $N_{K} H' / Z_{K} H' = 1$. Therefore, the two Borel subgroups containing $H' = K$ are fixed by $K$, and hence there are two distinct choices of $\Delta^+$, namely $ \{\alpha'\} = \{ i(e_{11}^* - e_{22}^*) + (e_{12}^* + e_{21}^*)\}$ or $\{ - \alpha'\}$. Consequently, the orbit representatives for this $H'$ are the triples 
\[(H', B_{\alpha'}, 0) \quand (H', B_{-\alpha'}, 0),\]
 where $B_{\alpha'}$ and $B_{-\alpha'}$ are $\theta$-stable Borel subgroups of $G$ corresponding to $\Delta^+ = \{\alpha'\}$ and $\Delta^+ = \{-\alpha'\}$ respectively. The two triples correspond to the double coset representatives $\begin{psmallmatrix}0 & i \\ i & 0 \end{psmallmatrix}$ and $\begin{psmallmatrix}0 & -i \\ -i & 0 \end{psmallmatrix}$, respectively.

The above example is somewhat complicated even for the simple cases with explicit matrix groups and small dimensions. 
Mann's results have the advantage that the classification of $LK$-orbits is type-independent. 
However, the classification given by Mann is non-constructive. It is not clear what the double coset representatives look like for general cases.

In this paper, although we limit our scope to $G = \GL_n$ and type A affine flag varieties, the double coset representatives $K\backslash G/I$ are explicitly given in matrix forms. They are also related to certain concrete affine (twisted) involutions under explicit bijections. These concrete affine twisted involutions open the door to determining the weak order of the ``closure" of the double cosets in $K\backslash G/I$, which would generalise the results of \cite{CJW}. 
\section{Orbits of orthogonal groups} \label{chpt3-ref}

In this section we study the orbits of orthogonal groups in the affine flag variety. In Section~\ref{sect3.1-ref}, we briefly review the strategy for finding double coset representatives of $K\backslash G/B$ with $(G,K) = (\GL_n(\CC), \O_n(\CC))$ and $B$ the standard Borel subgroup of $G$. We also investigate the case $(G, K) = (\GL_n(F), \O_n(F))$ in the same section, where $F = \KLt$. We then investigate the case $(G, K) = (\SL_n(F), \SO_n(F))$ in Section~\ref{sect3.3-ref}.

\subsection{Orbits of the orthogonal group}\label{sect3.1-ref}
The following decomposition in the finite case was known in \cite[Theorem 4.1]{MO} and \cite[Example 10.2]{Richardson90}:
\[\GL_n(\CC) = \bigsqcup_{w = w^{-1} \in S_n} \O_n(\CC)g_wB,\]
where $g_w$ is a certain element in $\GL_n(\CC)$ satisfying the property $g_w^T g_w = w$, with $w$ treated as a symmetric permutation matrix.

Now we are going to consider the affine version of the above decomposition. Before this, we will review how the above decomposition can be proven in the finite case \cite[Example 10.2]{Richardson90}. 
Suppose $k$ is a field with characteristic not equal to two. Denote the $n$-by-$n$ identity matrix by $1_n$.
\ben
\item[(i)] Consider the right action of $\GL_n(k)$ on itself, given by $x \cdot g:= g^Txg$. The stabilizer subgroup for $x = 1_n$ under this action consists of the elements $g \in \GL_n(k)$ such that $1_n = g^Tg$, i.e. the orthogonal group $\O_n(k)$. Therefore by the orbit-stabilizer theorem, the map $\O_n(k) g \mapsto g^Tg$ is a bijection between the set of right cosets $\O_n(k)\backslash \GL_n(k)$ and the set $\{g^Tg: g \in \GL_n(k)\}$.
\item[(ii)] Define a right $B$-action on the set $X = \{g^Tg: g \in \GL_n(k)\}$ by $g^Tg \cdot b := b^Tg^Tgb$ for $b \in B$. This right action commutes with the right $B$-action on $\O_n(k)\backslash \GL_n(k)$ given by right multiplication under the bijection in (i).
\item[(iii)] Hence, there is a bijection between the set $\O_n(k)\backslash \GL_n(k)/B$ and the $B$-orbits on the set $X$. 
\een

In the case where $k = \CC$, the set $X$ is the set of all symmetric matrices in $\GL_n(\CC)$. This can be seen as a corollary of the Autonne-Takagi factorization \cite[Corollary 4.4.4(c)]{horn2012matrix}. Therefore the $B$-orbits on $X$ are the symmetric $n$-by-$n$ permutation matrices, and hence $\O_n(\CC)\backslash \GL_n(\CC)/B$ is indexed by the involutions as described.

The affine case where $k = \KLt$ is more complicated. Recall that everywhere in this paper $\K$ is a quadratically closed field. From now until the end of this subsection, we assume $F = \KLt$, $G = \GL_n(F)$ and $K = \O_n(F)$. Define $B$ to be the Iwahori subgroup instead of the Borel subgroup. Now, the orbit $1_n \cdot G =\{g^Tg: g \in \GL_n(F)\}$ no longer consists of all symmetric matrices. Yet we can classify what exactly are the elements in $1_n \cdot G$. We shall explore the case $n=2$ first.

We first state a lemma concerning the squares in $F = \KLt$ and $A = \KPt$.

\begin{lemma}
An element $g \in F$ is a square if and only if $\ord(g)$ is even.
\begin{proof}
Suppose $\ord(g)$ is even, so that $g = t^{2m} u$ where $\ord(g) = 2m$ for some integer $m$, and $u \in A^*$. Note that $u$ can be written as $u_0(1+ tu')$ for some $u' \in \KPt$, where $u_0$ is the constant term of $u$. Because $\K$ is a quadratically closed field, the element $u_0$ has a well-defined square root. Therefore $u$ also has a square root $u^{1/2} = u_0^{1/2}(1+tu')^{1/2}$ \cite[Theorem 4.1]{Sambale} and it holds that $g =( t^m u^{1/2})^2$.

Conversely suppose $g$ is a square. Then $g = h^2$ for some $h \in \KLt$, and so $\ord(g) = 2 \ord(h)$, which is even.
\end{proof}
\end{lemma}

Before we consider the case $n=2$ and general $n \in \ZZ_{\geq 0}$, we first present results on the generators of the Iwahori subgroup $B_n$. 

For $m \in \ZZ$, we define $U_{ij,m} := 1_n + t^mA E_{ij}$. Let $T(A)$ be the group of $n$-by-$n$ diagonal matrices with all entries in $A$. Then the Iwahori subgroup $B$ of $\GL_n(A)$ can be written as \cite[Theorem 2.5]{IM}:
\[ B = {T}(A) \prod_{i<j}U_{ij,0} \prod_{i>j} U_{ij,1}.\]

In particular for $n=2$, the generators of $B_2$ are
\[ \begin{pmatrix} \alpha &0 \\ 0 & \beta \end{pmatrix} \text{ with } \alpha, \beta \in A^*, \quad \begin{pmatrix} 1 & t^k u \\ 0 & 1 \end{pmatrix} \text{ with } k\geq 0, u \in A^* , \quand \begin{pmatrix} 1 & 0 \\ t^k u & 1 \end{pmatrix} \text{ with } k\geq 1, u \in A^*. \]

The following fact will be used throughout this paper. Notice that multiplying elements $u_{ij}^T$ on the left of a matrix, where $u_{ij} \in U_{ij,0}$ or $U_{ij,1}$ is the same as doing row operations, either adding a $A$-multiple of a row to another row below, or adding a $tA$-multiple of a row to another row above. Similarly, multiplying the corresponding elements $u_{ij} \in U_{ij,0}$ or $U_{ij,1}$ on the right of a matrix is the same as doing column operations, either adding a $A$-multiple of a column to another column on the right, or adding a $tA$-multiple of a column to another column on the right.

\begin{lemma} \label{charsym1}
If $h = h^T \in \GL_n(F)$, then there exists $g \in \GL_n(F)$ such that 
 $g^Thg $ is a symmetric affine permutation matrix with only $1$'s in the off-diagonal entries, and $1$'s or $t$'s in the diagonal entries. 
 \end{lemma}
One example of the specified form of $g^Thg$, where $h = h^T \in \GL_6(F)$ is the following: 
 \[
 \begin{pmatrix}
 1& 0 & 0 & 0 & 0 & 0\\
 0 & 0 & 0 & 1 & 0 & 0\\
 0 & 0 & t & 0 & 0 & 0\\
 0 & 1 & 0 & 0 & 0 & 0\\
 0 & 0 & 0 & 0 & 0 & 1\\
 0 & 0 & 0 & 0 & 1 & 0
 \end{pmatrix}.
 \]

\begin{proof}
Consider the first row and first column of $h$, which are transposes of each other since $h = h^T$. There must be a non-zero entry in the first row and column since $h \in \GL_n(F)$. 
Suppose $h_{j1} = h_{1j} \neq 0$ are the uppermost and leftmost non-zero entries in first column and first row respectively.

By performing row operations on $h$, we can eliminate all the entries below $h_{j1}$. Similarly, by performing corresponding column operations, we can eliminate all the entries to the right of $h_{1j}$. This is equivalent to calculating $g^Thg$ where $g$ is a product of elementary matrices. The elimination process takes the following form, for example:
\[ \begin{pmatrix} 1 & 0 & 0 \\ 0& 1 & 0 \\ 0 & -b/a & 1 \end{pmatrix} \begin{pmatrix} 0& a & b \\ a & c & d\\ b &d & * \end{pmatrix} \begin{pmatrix}1 & 0 & 0\\ 0& 1 & -b/a \\ 0 &0 & 1 \end{pmatrix} 
= \begin{pmatrix} 0& a & 0 \\ a& c & d - bc/a \\0 &d - bc/a & *  \end{pmatrix},  \]
where different $*$'s are different elements in $F$, not written explicitly.  

Suppose $j=1$.
We can eliminate each entry below and to the right of $h_{11}$ to get all zeros in the first column and row. The resulting matrix is equal to $(g')^T hg'$ for some $g' \in \GL_n(F)$. Write $h_{11}= t^{b_{11}}u_{11}$ for some integer $b_{11}$ and $u_{11} \in A^*$. We then conjugate the resulting matrix $g'^T hg'$ by 
\[d:=\diag(t^{-\lfloor b_{11}/2 \rfloor}u_{11}^{-1/2}, 1, 1, \dots, 1).\]
 As a result, the $(1,1)$ entry of the resulting matrix $(g'd)^Th(g'd)$ is either $1$ or $t$, depending on whether $b_{11}$ is even or odd, respectively. We can then proceed similarly with the second row and column of $(g'd)^Th(g'd)$, since the $(1,2)$ and $(2,1)$ entries are zero.

Suppose, otherwise, $j > 1$.
We can repeat the process in the previous paragraph to eliminate every entry below and to the right of entries $h_{j1}$ and $h_{1j}$, except $h_{jj}$. To eliminate $h_{jj}$ it suffices to consider a generalisation of the following example with $n =2$:
\[ \begin{pmatrix} 1 & 0 \\ -c/2a &1 \end{pmatrix} \begin{pmatrix} 0 & a \\ a &c \end{pmatrix} \begin{pmatrix}1 & -c/2a \\ 0&1 \end{pmatrix} = \begin{pmatrix}  0 & a \\ a & c/2 \end{pmatrix} \begin{pmatrix}1 & -c/2a \\ 0&1 \end{pmatrix} = \begin{pmatrix} 0 & a \\ a & 0 \end{pmatrix}  \]
The resulting matrix is again equal to $g'^T hg'$ for some $g' \in \GL_n(F)$ being a product of elementary matrices.
From here, we conjugate the resulting matrix $g'^T hg'$ by $d = \diag(h_{j1}^{-1}, 1, 1, \dots, 1)$. As a result, the $(j,1)$ and $(1,j)$ entries of $(g'd)^T h(g'd)$ are $1$, i.e., we have $\begin{psmallmatrix} 0 & 1 \\ 1 & 0 \end{psmallmatrix}$ in the submatrix $(g'd)^T h(g'd)\big|_{\{1,j\} \times \{1, j \}}$.  We can then proceed similarly with the second row and column of $(g'd)^Th(g'd)$, observing the entries $(1,2)$, $(j,2)$, $(2,1)$ and $(2,j)$ are all zero.

By performing induction on the number of rows and columns, we can take $g$ equal to the product of the matrices described above for each row and column. We get the desired form of $g^Thg$, that is, a symmetric affine permutation matrix with only $1$'s in the off-diagonal entries, and $1$'s or $t$'s in the diagonal entries. 
\end{proof}
By using Lemma~\ref{charsym1} above, we derive a characterization of the elements in the set $\{g^Tg: g \in \GL_n(F)\}$ as provided by the following lemma.

\begin{lemma}\label{gtransposeg_n} \label{charsym2}
Suppose $h^T = h \in \GL_n(F)$.
Then $h \in \{g^Tg: g \in \GL_n(F)\} $
if and only if $\det (h)$ is a square in $F$.
\begin{proof}
If $h = g^Tg$, then $\det (h) = \det (g^T)\det(g) = \det(g)^2$. 

Conversely suppose $\det(h)$ is a square in $F$. Notice $h \in \{g^Tg: g \in \GL_n(F) \} $ if and only if there exists $g \in \GL_n(F)$ such that $g^Thg = 1$. 

By Lemma~\ref{charsym1} there exists $g$ such that $g^Thg$ is a symmetric affine permutation matrix with only $1$'s in the off-diagonal entries, and $1$'s or $t$'s in the diagonal entries.

Since the determinant of $h$ and hence $g^Thg$ is a square in $F$, the number of $t$'s in $g^Thg$ is even. Therefore we can pair up the $t$'s in $g^Thg$, so that for any pair of $t$'s located at the $(x,x)$ and $(y,y)$ entries of $g^Thg$, the submatrix $g^Thg\big|_{\{x,y\} \times \{x,y\}}$ is equal to $\begin{psmallmatrix} t& 0 \\ 0 & t\end{psmallmatrix}$. Fix such a pairing. We can also pair up the non-diagonal $1$'s located in the $(x,y)$ and $(y,x)$ entries of $g^Thg$, so that the submatrix $g^Thg\big|_{\{x,y\} \times \{x,y\}}$ is equal to $\begin{psmallmatrix} 0& 1 \\ 1 & 0 \end{psmallmatrix}$. Now notice that
\[  \begin{pmatrix} t& 0 \\ 0 & t\end{pmatrix} =  \begin{pmatrix} 1&  (t-1)^{1/2}\\  -(t-1)^{1/2} & 1 \end{pmatrix}\begin{pmatrix} 1&-(t-1)^{1/2} \\ (t-1)^{1/2} & 1 \end{pmatrix}, \] 
and 
\[ \begin{pmatrix} 0 &1 \\ 1& 0 \end{pmatrix}  = h^Th, \quad \text{ where } h = \begin{pmatrix} i &-\frac{1}{2} i \\1 & \frac{1}{2} \end{pmatrix}.\]
Therefore $g^Thg= X^T X$, where $X$ is formed from $g^Thg$ by replacing all pairs of $t$ as submatrices  $\begin{psmallmatrix} t& 0 \\ 0 & t\end{psmallmatrix}$ in $g^Thg$ by $\begin{psmallmatrix} 1&-(t-1)^{1/2} \\ (t-1)^{1/2} & 1 \end{psmallmatrix}$, and all the submatrices $\begin{psmallmatrix} 0 &1 \\ 1& 0 \end{psmallmatrix}$ in $g^Thg$ by $\begin{psmallmatrix} i &-\frac{1}{2} i \\1 & \frac{1}{2} \end{psmallmatrix}$. 
Therefore 
\[h = (Xg^{-1})^T (Xg^{-1}) \in  \{g^Tg: g \in \GL_n(F)\}.\]
\end{proof}
\end{lemma}
As before by the orbit-stablizer theorem, the map $\O_2(F)g \mapsto g^Tg$ is a bijection between 
$\O_2(F) \backslash \GL_2(F)$ and $ \{g^Tg: g \in \GL_2(F)\}$.
We want to find all the distinct orbits of $\O_2(F) \backslash \GL_2(F)/B_2$ by considering orbits of the right action of $b \in B_2$ acting on $x \in \{g^Tg : g \in \GL_2(F)\}$ defined by $x \cdot b := b^Txb$. The same right action of $b \in B_2$ can act on the set of all symmetric matrices in $\GL_2(F)$ as well.

The following proposition serves two purposes. First, the proposition provides the inductive step that will be used in the proof of Theorem~\ref{ortho-orbit-thm-intro}. Second, the proposition gives a list of orbit representatives in $\{g^Tg : g \in \GL_2(F)\}$ for this $B_2$-action. We will show later that these representatives belong to different orbits.

\begin{proposition}
\label{orthogonaln=2}
Each $B_2$-orbit in the set of symmetric matrices in $\GL_2(F)$ contains at least one element of the form
 $\begin{pmatrix} t^{m_1} & 0 \\ 0 & t^{m_2} \end{pmatrix}$ with $m_1,m_2 \in \ZZ$ or $\begin{pmatrix} 0& t^k \\ t^k & 0 \end{pmatrix}$ for $k \in \ZZ$.

\begin{proof}
Notice that for any symmetric matrix $\begin{psmallmatrix} a&c \\ c & d \end{psmallmatrix} \in \GL_n(F)$ and any $\alpha, \beta \in F^*$,
\begin{align}\label{ortho-remove-units}
 \begin{pmatrix} \alpha &0 \\ 0 & \beta \end{pmatrix}\begin{pmatrix} a&c \\ c & d \end{pmatrix} \begin{pmatrix} \alpha &0 \\ 0 & \beta \end{pmatrix} = \begin{pmatrix} a\alpha^2 &c\alpha\beta \\ c\alpha\beta & d\beta^2 \end{pmatrix}.
 \end{align}
For any $p(t) \in F$, write $p(t) = t^{\ord(p)} u_p$ with $u_p \in A^* = A \backslash tA$.
Set $m_1 = \ord(a)$, $m_2 = \ord(d)$. 
If $a$ or $d$ is $0$, we can set $(\alpha, \beta)  = (1, u_d^{-1/2})$ or $(u_a^{-1/2} ,1)$ respectively to assume $a = t^{m_1}$ or $d = t^{m_2}$ respectively. Also assume $k = \ord(c)$.

For any symmetric matrix $\begin{psmallmatrix} a&c \\ c & d \end{psmallmatrix} \in \{g^Tg: g\in \GL_2(F) \}$, we want to simplify the matrix as much as possible by a series of right $B_2$-actions. We divide our investigation into cases, with the results summarized in Table~\ref{table-ortho-rep}. 
\ben
\item Suppose $c = 0$. Then by putting $(\alpha, \beta)  = (u_a^{-1/2}, u_d^{-1/2})$ in the left hand side of \eqref{ortho-remove-units}, the right hand side equals to $\begin{psmallmatrix} t^{m_1} & 0 \\ 0 & t^{m_2} \end{psmallmatrix}$.

\item Suppose $ c \neq 0$ and $a = d = 0$. In this case put $(\alpha, \beta)  = (1, u_b^{-1/2})$ in the left hand side of \eqref{ortho-remove-units}, and the right hand side is equal to $\begin{psmallmatrix} 0 & t^k \\ t^k & 0\end{psmallmatrix}$.

\item Suppose $c \neq 0$, $d=0$ and $a \neq 0$. We first put $(\alpha, \beta)  = (u_a^{-1/2}, 1)$ in the left hand side of \eqref{ortho-remove-units} to obtain the matrix $\begin{psmallmatrix}t^{m_1} & cu_a^{-1/2} \\ cu_a^{-1/2} & 0 \end{psmallmatrix}$ on the right hand side. 
Denote $\gamma =cu_a^{-1/2}$. Then putting $(\alpha, \beta)  = (1, u_{\gamma}^{-1})$ and replacing $\begin{psmallmatrix} a&c \\ c & d \end{psmallmatrix}$ by $\begin{psmallmatrix}t^{m_1} & cu_a^{-1/2} \\ cu_a^{-1/2} & 0 \end{psmallmatrix}$ on the left hand side of \eqref{ortho-remove-units}, the right hand side equals to $\begin{psmallmatrix} t^{m_1} & t^k \\ t^k & 0\end{psmallmatrix}$.
\ben
\item If $k \geq m_1$, then $b^T \begin{psmallmatrix} t^{m_1} & t^k \\ t^k & 0\end{psmallmatrix} b = \begin{psmallmatrix} t^{m_1} & 0 \\ 0 & t^{2k-m_1} \end{psmallmatrix}$, where $b = \begin{psmallmatrix} 1 & -t^{k-m_1} \\ 0 & 1\end{psmallmatrix} \in B_2$.

\item If $k < m_1$, then $b^T \begin{psmallmatrix} t^{m_1} & t^k \\ t^k & 0\end{psmallmatrix} b = \begin{psmallmatrix}0 & t^k \\ t^k & 0 \end{psmallmatrix}$, where $b = \begin{psmallmatrix} 1 & 0\\ -\frac{1}{2} t^{m_1-k} & 1\end{psmallmatrix} \in B_2$.
\een
\item Suppose $c \neq 0$, $d\neq 0$ and $a=0$. The situation is symmetric to case 3. Using the equation \eqref{ortho-remove-units} to transform the matrix $\begin{psmallmatrix} a&c \\ c & d \end{psmallmatrix}$ into $\begin{psmallmatrix} 0 & t^k \\ t^k & t^{m_2}\end{psmallmatrix}$.
\ben
\item If $k > m_2$, then $b^T \begin{psmallmatrix} 0 & t^k \\ t^k & t^{m_2}\end{psmallmatrix} b = \begin{psmallmatrix} t^{2k-m_2} & 0 \\ 0 & t^{m_2} \end{psmallmatrix}$, where $b = \begin{psmallmatrix} 1 & 0 \\ -t^{k-m_2}  & 1\end{psmallmatrix} \in B_2$.
\item If $k \leq  m_2$, then $b^T \begin{psmallmatrix} 0 & t^k \\ t^k & t^{m_2}\end{psmallmatrix} b = \begin{psmallmatrix} 0 & t^k \\ t^k & 0 \end{psmallmatrix}$, where $b = \begin{psmallmatrix} 1 & -\frac{1}{2} t^{m_2-k}\\  0& 1\end{psmallmatrix} \in B_2$.
\een

\item Suppose $a, b, d \neq 0$. We first put $(\alpha, \beta)  = (u_a^{-1/2}, u_d^{-1/2})$ in the left hand side of \eqref{ortho-remove-units} to obtain the matrix $\begin{psmallmatrix}t^{m_1} & \gamma \\ \gamma & t^{m_2} \end{psmallmatrix}$ on the right hand side, where $\gamma = cu_a^{-1/2}u_d^{-1/2}$. 
Apply $\diag( u_a^{-1/2}, u_d^{-1/2})$ first to replace $d$ by $t^n$, $b$ by $\beta =bu_a^{-1/2}u_d^{-1/2}$, $a$ by $t^m$. 
We further divide the investigation into several cases:
\ben
\item Suppose $m_1 \geq m_2$. 
\ben
\item Suppose $\ord (\gamma) \geq m_1 \geq m_2$. then $b^T \begin{psmallmatrix} t^{m_1} & \gamma \\ \gamma & t^{m_2} \end{psmallmatrix} b = \begin{psmallmatrix}  t^{m_1} & 0 \\ 0 & t^{m_2} - \gamma^2/t^{m_1} \end{psmallmatrix}$, where $b = \begin{psmallmatrix} 1 & -\gamma / t^{m_1} \\ 0  & 1\end{psmallmatrix} \in B_2$.
Applying equation \eqref{ortho-remove-units} with $(\alpha, \beta)  = (1, u_{t^{m_2} - \gamma^2/t^{m_1}})$ and $\begin{psmallmatrix}  t^{m_1} & 0 \\ 0 & t^{m_1} - \gamma^2/t^{m_1} \end{psmallmatrix}$ in place of  $\begin{psmallmatrix}  a & c \\ c& d\end{psmallmatrix}$ on the left hand side to obtain $\begin{psmallmatrix} t^{m_1} & 0 \\ 0 & t^{m_3} \end{psmallmatrix}$ on the right hand side, where $m_3 = \ord(t^{m_2} - \gamma^2/t^{m_1})$.

Suppose $\ord (\gamma) \leq m_2 \leq m_1$. Observe that 
\[ \begin{pmatrix} 1 & 0 \\ p & 1\end{pmatrix}\begin{pmatrix} t^{m_1} &\gamma \\ \gamma & t^{m_2} \end{pmatrix} \begin{pmatrix} 1 &p   \\ 0 & 1 \end{pmatrix} = \begin{pmatrix} t^{m_1} &\gamma + pt^{m_1}\\ \gamma + pt^{m_1} & p^2t^{m_1} + 2p\gamma + t^{m_2} \end{pmatrix}. \] 

We want to eliminate the $(2,2)$ entry, so we set $p$ as unknown and solve the equation $p^2t^{m_1} + 2p\gamma + t^{m_2} = 0$. We have \[p = (-2\gamma \pm (4\gamma^2-4t^{m_1+ m_2})^{1/2})/2t^{m_1}.\]
 But notice that 
\begin{align*}
(4\gamma^2-4t^{m_1+ m_2})^{1/2} &= 2\gamma (1-t^{m_1+ m_2}/\gamma^2)^{1/2}\\
& = 2\gamma (1- t^{m_1+ m_2}/2\gamma^2 + \dots) \\
&= 2\gamma - t^{m_1+ m_2}/\gamma + \dots
\end{align*}

Taking the plus sign in the formula for $p$, we have $p = -t^{m_2}/ (2\gamma) +$ higher degree terms. Therefore $p$ is in $A$ with order $m_2-\ord(\gamma) \geq 0 $, and hence $\begin{psmallmatrix} 1 &p   \\ 0 & 1 \end{psmallmatrix} \in B_2$. 
Let $\delta := \gamma + pt^{m_1}$. Depending on whether $\ord(\delta) < m_1$ or $\ord(\delta) \geq m_1$, applying right $B_2$-actions on the matrix $\begin{psmallmatrix} t^{m_1} & \gamma + pt^{m_1} \ \gamma + pt^{m_1} & 0 \end{psmallmatrix}$ results in one of the following elements in $\GL_2(F)$:
\[ \begin{pmatrix}0& t^{\ord (\delta) } \\  t^{\ord (\delta)}& 0 \end{pmatrix}  \text{ or }  \begin{pmatrix} t^{m_1} & 0 \\ 0 & t^{2\ord(\delta)-m_1} \end{pmatrix}. \]
from $\begin{psmallmatrix} t^{m_1} &\gamma + pt^{m_1}\\ \gamma + pt^{m_1} &0 \end{psmallmatrix}$ using the equation \eqref{ortho-remove-units} and a right $B_2$-action as stated in case 3(a) or 3(b), depending on whether $\ord(\delta) < m_1$ or $\ord(\delta) \geq m_1$.

\item Suppose $m_1 > \ord(\gamma) >  m_2$.  Then $b^T \begin{psmallmatrix} t^{m_1} & \gamma \\ \gamma & t^{m_2} \end{psmallmatrix} b = \begin{psmallmatrix} t^{m_1} - \gamma^2/t^{m_2} & 0 \\ 0  & t^{m_2}\end{psmallmatrix}$, where 
\[b = \begin{psmallmatrix} 1 & 0  \\ -\gamma / t^{m_2}  & 1\end{psmallmatrix} \in B_2.\]

Using case (1), we can reduce the resulting matrix to $\begin{psmallmatrix} t^{m_4} & 0 \\ 0 & t^{m_2} \end{psmallmatrix}$, where $m_4 = \ord( t^{m_1} - \gamma^2/t^{m_2} )$.
\een

\item Suppose $m_1<m_2$. This case is symmetric to case 5(a), so we just list the results in Table~\ref{table-ortho-rep}.
\een
\een
Notice that $\ord (\gamma) = \ord(cu_a^{-1/2}u_d^{-1/2})  = \ord(c) = k$. 
Table~\ref{table-ortho-rep} summarises all the cases.
\begin{table}[htbp]
\begin{center}
\begin{tabular}{ c c  }
Distinct cases for $a, c, d$ in $\left(\begin{smallmatrix} a & c \\ c & d \end{smallmatrix}\right)$ & Orbit representatives    \\ 
$c = 0$ & $\begin{psmallmatrix} t^{m_1} & 0 \\ 0 & t^{m_2} \end{psmallmatrix}$  \\  
$c \neq 0$, $a = d = 0$ & $\begin{psmallmatrix} 0 & t^k \\ t^k & 0 \end{psmallmatrix}$   \\
$a, c \neq 0$, $d = 0$, $k \geq m_1$ & $\begin{psmallmatrix} t^{m_1} & 0 \\ 0 & t^{2k-m_1} \end{psmallmatrix}$  \\
$a, c \neq 0$, $d = 0$, $k < m_1$ & $\begin{psmallmatrix} 0 & t^k \\ t^k & 0 \end{psmallmatrix}$ \\
$c, d \neq 0$, $a = 0$, $k > m_2$ & $\begin{psmallmatrix} t^{2k-m_2} & 0 \\ 0 & t^{m_2} \end{psmallmatrix}$ \\
$c, d \neq 0$, $a = 0$, $k \leq m_2$ & $\begin{psmallmatrix} 0 & t^k \\ t^k & 0  \end{psmallmatrix}$ \\
$a, c, d \neq 0$, $k \geq m_1 \geq m_2$ & $\begin{psmallmatrix} t^{m_1} & 0 \\ 0 & t^{m_3} \end{psmallmatrix}$\\
$a, c, d \neq 0$, $ m_1 \geq m_2 \geq k$  &  $\begin{psmallmatrix} 0& t^{\ord (\delta) } \\  t^{\ord (\delta)}\end{psmallmatrix}$ or $\begin{psmallmatrix} t^{m_1} & 0 \\ 0 & t^{2\ord(\delta)-m_1} \end{psmallmatrix}$ \\
$a, c, d \neq 0$, $m_1 > k > m_2$  & $\begin{psmallmatrix} t^{m_3} & 0 \\ 0 & t^{m_2} \end{psmallmatrix}$ \\
$a, c, d \neq 0$, $m_1< m_2\leq k$ & $\begin{psmallmatrix} t^{m_4} & 0 \\ 0 & t^{m_2} \end{psmallmatrix}$ \\
$a, c, d \neq 0$, $k < m_1< m_2 $ & $\begin{psmallmatrix} 0 & t^k \\ t^k & 0 \end{psmallmatrix}$\\
$a, c, d \neq 0$, $m_1\leq k < m_2 $ & $\begin{psmallmatrix} t^{m_1} & 0 \\ 0 & t^{m_3} \end{psmallmatrix}$
\end{tabular}
\caption{Orbit representatives of the right $B_2$-action on the set $\{g^Tg: g\in \GL_2(F) \}$. Here $m_1 = \ord(a)$, $m_2 = \ord(d)$, $k = \ord (c)$, $\gamma = cu_a^{-1/2}u_d^{-1/2}$, $\delta= \gamma + pt^{m_1}$ where $p \in F$ is a root of certain quadratic equation in $F$, $m_3 = \ord(t^{m_2} - \gamma^2/t^{m_1})$ and $m_4 = \ord( t^{m_1} - \gamma^2/t^{m_2} )$. }\label{table-ortho-rep}
\end{center}
\end{table}
\end{proof}
\end{proposition}
The above proposition together with Lemma~\ref{charsym2} implies the following corollary.
\begin{corollary}
Each $B_2$-orbit in the set $\{g^Tg: g\in \GL_2(F)\}$ contains at least one element of the form
 $\begin{pmatrix} t^{m_1} & 0 \\ 0 & t^{m_2} \end{pmatrix}$ with $m_1,m_2 \in \ZZ$, $m_1 + m_2$ even, or $\begin{pmatrix} 0& t^k \\ t^k & 0 \end{pmatrix}$ for $k \in \ZZ$.
 \end{corollary}
We have not yet shown that $\begin{pmatrix} t^m & 0 \\ 0 & t^n \end{pmatrix}$ with $m,n \in \ZZ$, $m+n$ even, and $\begin{pmatrix} 0& t^n \\ t^n & 0 \end{pmatrix}$ for $n \in \ZZ$ are in different orbits under the right action of $B_2$. We will prove this result for general $n$ in the proof of Theorem~\ref{ortho-orbit-thm-intro}.

Recall that we identify each element in the extended affine symmetric group $\widetilde S_n^+$ as an affine permutation matrix; see the paragraph before Remark~\ref{affine-perm-as-matrix}. We also use the notation $\tau^{\mathbf{c}} = \diag(t^{c_1}, t^{c_2}, \dots, t^{c_n})$ for $\mathbf{c} = (c_1, c_2, \dots, c_n)$.

We now recall the necessary definitions for the combinatorial indexing set for the double cosets in $\O_n(F) \backslash \GL_n(F) / B_n$. 
Recall from Definition~\ref{def-extended-aff-twisted-inv-1} that the set $\affinvol_n$ consists of all symmetric $n$-by-$n$ affine permutation matrices, and the set $\eaffinvol_n$ consists of all elements in $\affinvol_n$ for which  the sum of the powers of $t$ is even.
\begin{remark}
Define $^*$ to be the automorphism on $\bar{w}\tau^{\mathbf{c}} \in \widetilde S_n^+$ by $\bar{w}\tau^{\mathbf{c}} \mapsto \bar{w}\tau^{\mathbf{-c}}$. Then the set $\affinvol_n$ consists of all affine permutations $w$ with the property $w^* = w^{-1}$, and therefore we can call these affine permutations as \defn{extended affine twisted involutions}. Similarly, the set $\eaffinvol_n$ consists of all affine permutations in $\affinvol_n$ for which $c_1+c_2+\dots+c_n$ is an even integer. Therefore we can also refer to elements in $\eaffinvol_n$ as \defn{even extended affine twisted involutions}.
\end{remark}

We shall need the following lemma to show that different elements in $\eaffinvol_n$ lie in different orbits of a certain right $B_n$-action on the set $\{g^Tg : g \in G\}$.

\begin{lemma} \label{uniq-orbit-ortho}
Suppose $w, w' \in \eaffinvol_n$. Let $B_{n}'$ be the subgroup of lower triangular matrices modulo $t$ in $\GL_n(A)$. If $b_-wb = w'$ for some $b_- \in B_{n}'$ and $b \in B_n$, then $w = w'$.

\begin{proof}
By the affine Bruhat decomposition \cite[Theorem 2.16]{IM}, it holds that $w = w'$.
\end{proof}
\end{lemma}

Consider the following identities:
for $a,b$ being odd integers, it holds that
\begin{equation} \label{diag22blockbreak}
\begin{pmatrix} t^a& 0 \\ 0 & t^b\end{pmatrix} =  \begin{pmatrix} t^{\frac{a-1}{2}}&t^{\frac{a-1}{2}} (t-1)^{\frac{1}{2}}\\  -(t-1)^{\frac{1}{2}}t^{\frac{b-1}{2}} & t^{\frac{b-1}{2}} \end{pmatrix}\begin{pmatrix} t^{\frac{a-1}{2}}&-(t-1)^{\frac{1}{2}}t^{\frac{b-1}{2}} \\ t^{\frac{a-1}{2}} (t-1)^{\frac{1}{2}} & t^{\frac{b-1}{2}} \end{pmatrix}, 
\end{equation}
and for any integer $a$, it holds that
\begin{equation}\label{antidiag22blockbreak}
 \begin{pmatrix} 0& t^a \\ t^a & 0\end{pmatrix} = \begin{pmatrix} i&1 \\ -i t^a/2 & t^a/2 \end{pmatrix}  \begin{pmatrix} i& -it^a/2 \\ 1& t^a/2 \end{pmatrix}.
 \end{equation} 
\begin{definition}\label{def-gw-ortho}
Given $w = \bar{w} \tau^{\mathbf{c}} \in \eaffinvol_n$, we define $g_w$ as follows:
\ben
\item[(i)] If $c_i$ is even and $\bar{w}(i) = i$, replace each $t^{c_i}$ in $w$ with $t^{c_i/2}$,
\item[(ii)] For each pair $(i,j)$ with odd $c_i$ and $c_j$ and $\bar{w}(i) = i$, $\bar{w}(j) = j$, replace the $2$-by-$2$ submatrices $\{i,j\} \times \{i,j\}$ in $w$ with the matrix $\begin{pmatrix} t^{\frac{a-1}{2}}&-(t-1)^{\frac{1}{2}}t^{\frac{b-1}{2}} \\ t^{\frac{a-1}{2}} (t-1)^{\frac{1}{2}} & t^{\frac{b-1}{2}} \end{pmatrix}$.
\item[(iii)] For each pair $(c_i, c_j) = (c_j, c_i)$ where $(i,j)$ forms a cycle in $\bar{w}$, we replace the $2$-by-$2$ submatrices $\{i,j\} \times \{i,j\}$ in $w$ with the matrix $\begin{pmatrix} i& -it^a/2 \\ 1& t^a/2 \end{pmatrix}$.
\een
\end{definition}

As an example if $ w= \begin{psmallmatrix}
t^4& 0 & 0 & 0 & 0\\
0 & 0 & 0 & t^{-2} & 0\\
0 & 0 & t^{-5} & 0 & 0\\
0 & t^{-2} & 0 & 0 & 0\\
0 & 0 & 0 & 0 & t^3
\end{psmallmatrix} = (2 \, 4) \tau^{(4, -2, -5, -2, 3)}$, then
\[ g_w = \begin{pmatrix}
t^2& 0 & 0 & 0 & 0\\
0 & i & 0 & -it^{-2}/2 & 0\\
0 & 0 & t^{-3} & 0 & -(t-1)^{\frac{1}{2}} t\\
0 & 1 & 0 & t^{-2}/2 & 0\\
0 & 0 &  t^{-3}(t-1)^{\frac{1}{2}} & 0 & t
\end{pmatrix}.\]
By the identities \eqref{diag22blockbreak} and \eqref{antidiag22blockbreak}, for any $w \in \eaffinvol_n$, it holds that
\begin{equation}\label{sqrt-ortho-affineperm}
g_w^Tg_w = w.
\end{equation}
Recall that $G = \GL_n(F)$ and $K = \O_n(F)$.
For $w \in \eaffinvol_n$, define $\mathcal{O}_w = K g_w B$. Proving Theorem~\ref{ortho-orbit-thm-intro} is equivalent to showing that \[w \mapsto \mathcal{O}_w \]
is a bijection between $\eaffinvol_n$ and $ K\backslash G/ B$.

\begin{proof}[Proof of Theorem~\ref{ortho-orbit-thm-intro}]
The fact that (1) the quotient $K\backslash G$ is in bijection with the set $\{g^Tg : g \in G\}$, which is a subset of symmetric matrices in $G$ with square determinant (by Lemma~\ref{gtransposeg_n}), and (2) the bijection $Kg \mapsto g^Tg$ commutes with the right $B$-actions $Kg \cdot b = Kgb$ and $(g^Tg)\cdot b = b^Tg^Tgb$, is similar to the classical case explained in the beginning of Section~\ref{sect3.1-ref}.

Suppose $h \in \{g^Tg : g \in G\}$. Among the non-zero entries in $h$ with the least order, denote $h_{ij}$ to be the $(i,j)$ entry of $h$ that is the leftmost and upmost along its column and row (i.e., all the non-zero entries above and to the left of $h_{ij}$ have strictly greater order, and all the non-zero entries below and to the right of $h_{ij}$ have orders not less than that of $h_{ij}$).

Recall that conjugating a matrix $h$ by an elementary matrix $u \in B$, that is, computing $u^Th u$, is the same as performing simultaneous row and column operations. These operations can be either:
(i) consecutively adding a $F$-multiple of a row to another row below and then adding the same $F$-multiple of the corresponding column to the corresponding column on the right, or
(ii) consecutively adding a $tF$-multiple of a row to another row above and then adding the same $tF$-multiple of the corresponding column to the corresponding column on the left.

Suppose $i=j$. Write $h_{ii} = t^{\ord(h_{ii})} u_{ii}$, where $u_{ii} \in A^*$, and denote 
\[d = \diag(1, 1, \dots, 1, u_{ii}^{-1/2}, 1, \dots, 1),\]
 where $u_{ii}^{-1/2}$ is located at the $i$-th entry. Notice that $d$ is well-defined because $\K$ is quadratically closed. Then the $(i,i)$ entry of $dgd$ is $t^{\ord(h_{ii})}$, and the order of all other entries in $dgd$ is the same as those in $g$. We can then conjugate by a series of elementary matrices $u_{j} \in B$, with $j \in [1,m]$ for some $m\in \NN$, to eliminate all the non-zero entries in the $i$-th row and column, using $t^{\ord(h_{ii})}$.
The resulting matrix is 
\[u_m^T\dots u_2^Tu_1^Tdgdu_1u_2\dots u_m,\]
 and $t^{\ord(h_{ii})}$ is the only non-zero entry in the $i$-th row and column of this matrix. Deleting the $i$-th row and column of the matrix results in an $(n-1)$-by-$(n-1)$ invertible matrix. 

Suppose $i\neq j$. Without loss of generality, we assume $i>j$ since $h_{ij} = h_{ji}$. Similar to the previous paragraph, write $h_{ij} = t^{\ord(h_{ij})} u_{ij}$, where $u_{ij} \in A^*$, and denote 
\[d = \diag(1, 1, \dots, 1, u_{ij}^{-1}, 1, \dots, 1),\]
 where $u_{ij}^{-1}$ is located at the $i$-th entry. Then the $(i,j)$ and $(j,i)$ entries of $dgd$ are $t^{\ord(h_{ij})}$, and the order of all other entries in $dgd$ is the same as the corresponding entries in $g$. We can then conjugate by a series of elementary matrices $u_{k} \in B$, with $k \in [1,m]$ for some $m\in \NN$, to eliminate all the non-zero entries in the $i$-th and $j$-th rows and columns, using $t^{\ord(h_{ij})}$, except for the $(i,i)$ and $(j,j)$ entries. The resulting matrix is 
 \[u_m^T\dots u_2^Tu_1^Tdgdu_1u_2\dots u_m,\]
  and the only possible non-zero entries in the $i$-th and $j$-th rows and columns are the $(i,i)$, $(i,j)$, $(j,i)$, and $(j,j)$ entries. Therefore, we can use the recipe in Proposition~\ref{orthogonaln=2} to the $i$-th and $j$-th columns and rows of $u_m^T\dots u_2^Tu_1^Tdgdu_1u_2\dots u_m$ by further conjugation with elements in $B$. This allows us to obtain a resulting matrix in which either the $(i,i)$ and $(j,j)$ entries or the $(i,j)$ and $(j,i)$ entries are non-zero in the $i$-th and $j$-th columns and rows, and both non-zero entries are equal to a power of $t$. In particular, if the $(i,j)$ and $(j,i)$ entries are non-zero, then they are equal to the same power of $t$. Deleting the $i$-th and $j$-th rows and columns of this matrix results in an $(n-2)$-by-$(n-2)$ invertible matrix.

Regardless of whether $i = j$ or $i \neq j$, we can observe that the resulting matrices above can be conjugated by elementary matrices $u \in B$ without altering the $i$-th and $j$-th rows and columns. Thus, through induction on the matrix size, we deduce the existence of $b \in B$ such that $b^Thb$ is an element of $\affinvol_n$. Given that $h \in \{g^Tg: g \in G\}$, we can conclude that $b^Thb \in \eaffinvol_n$ by Lemma~\ref{charsym2}. Write $h = g^Tg$ for some $g\in G$, and denote $b^Thb \in \eaffinvol_n$ as $w_g$.  

The above construction proves the existence of $w_g$ for $K g B$ with $g \in G$, and the construction of $g_w$ right before the theorem gives the desired property of $g_w$ from a given $w \in \eaffinvol_n$. Thus, we have proven that the map $w \mapsto \mathcal{O}_w$ is a surjection.

It remains to show that the representatives are disjoint under the right $B$-action on the set $\{g^Tg : g\in G\}$. Suppose $\mathcal{O}_w = \mathcal{O}_{w'}$ for some $w, w' \in \eaffinvol_n$. By the proof above, we see that $g_{w'}^{T}g_{w'} \in \{b^Tg_{w}^Tg_{w}b : b \in B\}$. Therefore, $w' = b^Twb$ for some $b \in B$, and by Lemma~\ref{uniq-orbit-ortho}, we have $w' = w$. Hence, the elements $w$ in $\eaffinvol_n$ correspond to distinct double cosets $Kg_wB$.
\end{proof}

\subsection{Choices of $K$ and $B$}\label{sect3.2-ref}

In the previous section, we considered $K = \O_n(F)$ as the fixed-point subgroup of $G$ under the Cartan involution $\theta: G \rightarrow G$, defined as $\theta(g) = (g^T)^{-1}$. However, as in Richardson and Springer's paper \cite{Richardson90}, it is sometimes convenient to consider a different Cartan involution $\theta_1: G \rightarrow G$ defined as $\theta_1(g) = w_0 (g^T)^{-1} w_0$, where

\[ w_0 = \begin{pmatrix} & & & & 1\\
&&&1&\\
& & \iddots & & \\
& 1 & & & \\
1 & & & & \end{pmatrix}. \] 

The fixed-point subgroup of this $\theta_1$ is a conjugate of $K$. More precisely, it is the orthogonal subgroup $L = eKe^{-1}$, where $e$ is an $n$-by-$n$ matrix satisfying $e^2 = w_0$. 

Richardson and Springer introduced this twisted version of the Cartan involution because it has the additional property of fixing $B$, the Borel subgroup. The reason for insisting on the involution $\theta_1$ fixing $B$ is that $B$ determines a system of positive roots $\Delta^+$ in the root system $\Delta$ with respect to a $\theta_1$-stable maximal torus $T \subset B$, and $\theta \Delta^+ = \Delta^+$. This allows us to consider the $\theta$-action on the root system and root subgroups in $G$, which facilitates calculations involving the closure of orbits $KgB$. For more details, refer to \cite{Springer}.

It is worth noting that the double cosets in $K \backslash G/ B$ and $L \backslash G/ B$ are in bijection because $K$ and $L$ are conjugate subgroups of $G$.

\subsection{Orbits of the special orthogonal group}\label{sect3.3-ref}

In this section, we focus on the orbits of the special orthogonal group over the affine flag variety.

Let $\SL_n(F)$ and $\SO_n(F)$ denote the special linear group and the special orthogonal group over formal Laurent series, respectively. In this subsection we denote $G = \SL_n(F)$ and $K = \SO_n(F)$. In particular, $G$ and $K$ are the subgroups of matrices with a determinant equal to $1$ in $\GL_n(F)$ and $\O_n(F)$, respectively.

We now let $B=B_n$ be the standard Iwahori subgroup of $G$, which consists of matrices in $\SL_n(A)$ that are upper triangular modulo $t$. The method of finding double coset representatives in $K \backslash G/B$ is exactly the same as the case $\O_n(F) \backslash \GL_n(F)/B$.

As in the previous section, we first consider the case where $n = 2$. Recall that the right action of $b \in B_2$ acting on $x \in \{g^Tg : g \in \GL_2(F)\}$ defined by $x \cdot b := b^Txb$. 

The following proposition with similar proof as the Proposition~\ref{orthogonaln=2} will be used in the proof of Theorem~\ref{sortho-orbit-thm-intro}.
\begin{proposition}
\label{sorthogonaln=2-gl2}
When $B_2$ is the standard Iwahori subgroup of $\SL_2(F)$, each $B_2$-orbit in the set of all symmetric matrices in $\GL_2(F)$ contains at least one element of the form $\begin{pmatrix} t^m u & 0 \\ 0 & t^{n} \end{pmatrix}$ or $\begin{pmatrix} t^m & 0 \\ 0 & t^{n}u \end{pmatrix}$ with $m,n \in \ZZ$ and $u \in A^*$, or $\begin{pmatrix} 0& c  \\ c & 0 \end{pmatrix}$ with $c \in F^*$.
\end{proposition}
We state the proposition in this way for convenience, but notice that the first two matrices are in the same $B_2$-orbit.
\begin{proof}
Suppose $\begin{psmallmatrix} a&c \\ c & d \end{psmallmatrix} \in \GL_2(F)$.
Write $a = t^{m_1} u_a$, $c = t^k u_c$ and $d = t^{m_2} u_d$, where $u_a$, $u_c$ and $u_d \in A^*$. 
A special case of \eqref{ortho-remove-units} is:
\begin{align}\label{sortho-remove-units}
 \begin{pmatrix} \alpha &0 \\ 0 & \alpha^{-1} \end{pmatrix}\begin{pmatrix} a&c \\ c & d \end{pmatrix} \begin{pmatrix} \alpha &0 \\ 0 & \alpha^{-1} \end{pmatrix} = \begin{pmatrix} a\alpha^2 &c \\ c & d\alpha^{-2} \end{pmatrix}. 
 \end{align}

For any symmetric matrix $\begin{psmallmatrix} a&c \\ c & d \end{psmallmatrix} \in \GL_2(F)$, we want to simplify the matrix as much as possible by a series of right $B_2$-actions. We divide our investigation into cases.
\ben
\item Suppose $c = 0$.  By putting $\alpha  = u_d^{-1/2}$ in \eqref{sortho-remove-units}, the right hand side equals $\begin{psmallmatrix} t^{m_1} u_a u_d^{-1}& 0 \\ 0 & t^{m_2} \end{psmallmatrix}$.

\item Suppose $c \neq 0$ and $a = d = 0$. In this case there is nothing to show.
\item Suppose $c \neq 0$, $d=0$, $a \neq 0$. 
 We first put $\alpha  = u_a^{-1/2}$ in the left hand side of \eqref{sortho-remove-units} to obtain $\begin{psmallmatrix}t^{m_1} & c\\ c& 0 \end{psmallmatrix}$ on the right hand side. 
\ben
\item If $k \geq m_1$, then $b^T \begin{psmallmatrix}t^{m_1} &  c \\  c & 0\end{psmallmatrix} b = \begin{psmallmatrix}  t^{m_1} & 0 \\ 0 & t^{m_1+2k}(-u_c^2) \end{psmallmatrix}$, where $b = \begin{psmallmatrix}1 & -ct^{-m_1} \\ 0 & 1\end{psmallmatrix} \in B_2$.
\item If $k < m_1$, then $b^T \begin{psmallmatrix} t^{m_1} &  c\\ c & 0\end{psmallmatrix} b = \begin{psmallmatrix}0 & c \\ c & 0 \end{psmallmatrix}$, where $b = \begin{psmallmatrix} 1 & 0\\ \frac{-1}{2c} t^{m_1} & 1\end{psmallmatrix} \in B_2$.
\een
\item Suppose $c \neq 0$, $d\neq 0$ and $a=0$. The situation is similar to case 3. Put $\alpha  = u_d^{1/2}$ in the left hand side of \eqref{sortho-remove-units} to obtain $\begin{psmallmatrix}0 & c\\ c& t^{m_2} \end{psmallmatrix}$ on the right hand side. 
\ben
\item If $k > m_2$, then $b^T \begin{psmallmatrix} 0 & c \\ c & t^{m_2}\end{psmallmatrix} b = \begin{psmallmatrix} t^{2k-m_2}(-u_c^2) & 0 \\ 0 & t^{m_2} \end{psmallmatrix}$, where $b = \begin{psmallmatrix} 1 & 0 \\ -ct^{-m_2}  & 1\end{psmallmatrix} \in B_2$. 
\item If $k \leq  m_2$, then $b^T \begin{psmallmatrix} 0 & c \\ c & t^{m_2}\end{psmallmatrix} b = \begin{psmallmatrix} 0 & c \\ c & 0 \end{psmallmatrix}$, where $b = \begin{psmallmatrix} 1 & \frac{-1}{2c} t^{m_2}\\  0& 1\end{psmallmatrix} \in B_2$.
\een

\item Suppose $a, c, d \neq 0$. We can put $\alpha = u_a^{-1/2}$ or $u_d^{1/2}$ in the left hand side of \eqref{sortho-remove-units} to obtain the matrices $\begin{psmallmatrix}t^{m_1} & c \\ c & du_a \end{psmallmatrix}$ or $\begin{psmallmatrix}au_d & c \\ c & t^{m_2} \end{psmallmatrix}$ on the right hand side respectively.
We further divide the investigation into several cases:
\ben
\item Suppose $m_1 \geq m_2$. 
\ben
\item Suppose $k \geq m_1 \geq m_2$. then $b^T \begin{psmallmatrix} t^{m_1} & c \\ c & du_a \end{psmallmatrix} b = \begin{psmallmatrix}  t^{m_1} & 0 \\ 0 & du_a - c^2/t^{m_1} \end{psmallmatrix} = \begin{psmallmatrix}  t^{m_1} & 0 \\ 0 & t^{m_3}u \end{psmallmatrix} $, where $b = \begin{psmallmatrix} 1 & -c / t^{m_1} \\ 0  & 1\end{psmallmatrix} \in B_2$, $m_3 = \ord(du_a - c^2/t^{m_1} ) \in \ZZ$ and $u \in A^*$.
\item Suppose $m_1 \geq m_2 \geq k$. Similar to the proof of case 5(a)ii in Lemma~\ref{orthogonaln=2},
\[ \begin{pmatrix} 1 & 0 \\ p & 1\end{pmatrix}\begin{pmatrix} t^{m_1} &c \\ c & du_a \end{pmatrix} \begin{pmatrix} 1 &p  \\ 0 & 1 \end{pmatrix} = \begin{pmatrix} t^{m_1} &c + pt^{m_1} \\ c + pt^{m_1} & 0 \end{pmatrix}, \]
where $p = (-2c + (4c^2 - 4t^{m_1}du_a)^{1/2})/2t^{m_1} \in A$. Write $c + pt^{m_1} = t^{m_p}u_p$.   
Using case 3(a) or 3(b) above, the $B_2$-orbit containing $\begin{psmallmatrix} a & c \\ c & d\end{psmallmatrix}$ either contains $\begin{psmallmatrix} t^{m_1}& 0 \\ 0 & t^{m_1 + 2m_p} (-u_p^2)\end{psmallmatrix}$ or $\begin{psmallmatrix} 0 & c + pt^{m_1} \\ c + pt^{m_1} & 0 \end{psmallmatrix}$, depending on $m_1 \leq m_p$ or $m_1 > m_p$ respectively.

\item Suppose $m_1 > k>  m_2$.  then $b^T \begin{psmallmatrix}au_d & c \\ c & t^{m_2} \end{psmallmatrix} b = \begin{psmallmatrix} t^{m_1} - c^2/t^{m_2} & 0 \\ 0  & t^{m_2} \end{psmallmatrix} =\begin{psmallmatrix} t^{m_4}u & 0 \\ 0  & t^{m_2} \end{psmallmatrix}$, where $b = \begin{psmallmatrix} 1 &0  \\  -c / t^{m_2}  & 1\end{psmallmatrix} \in B_2$, $m_4 = \ord (t^{m_1} - c^2/t^{m_2}) \in \ZZ$ and $u \in A^*$. 
\een

\item Suppose $m_1<m_2$. The case study is similar to case 5(a), considering the subcases $k \geq m_1 > m_2$, $m_1 > k > m_2$ and $m_1 > m_2 \geq k$. Each symmetric matrix in $\SL_2(F)$ are in the same $B_2$-orbit with matrices of the form given in the description of the proposition. We omit the details here.
\een
\een
\end{proof}
We now give the necessary definition for the combinatorial indexing set for the double cosets in $\SO_n(F) \backslash \SL_n(F) / B = K \backslash G /B$.
Recall that $\widetilde{S_n}$ is the set of affine permutations $\{ \pi \in \widetilde S_n^+ : \sum^n_{i = 1} (\pi(i) - i) = 0 \}$, which is a subset of the extended affine permutations $\widetilde S_n^+$. Also recall that $\widetilde {\mathcal{I}_n}$ is the set of extended affine twisted involutions; see Definition~\ref{def-extended-aff-twisted-inv-1}.  

Define the set of \defn{affine twisted involutions} to be the set $\widetilde {\mathcal{I}_n^0} := \widetilde {\mathcal{I}_n}  \cap \widetilde{S_n}$. 
Therefore, the affine twisted involutions are the extended affine twisted involutions $w = \bar w \tau^{(c_1, c_2, \dots, c_n)}$ satisfying $c_1 + c_2 + \dots + c_n = 0$.

In case (i) of Definition~\ref{def-aff-tw-invol} where matrices in $i\affinvol_n$ contain at least one non-zero diagonal entries, each matrix is indexed by an element in $\widetilde {\mathcal{I}_n^0}$ using the following definition. 
\begin{definition} \label{hw}
For $w \in \widetilde {\mathcal{I}_n^0}$, define $h_w$ as the symmetric monomial matrix obtained from the affine permutation matrix of $w$ by multiplying every non-diagonal entry by $i$. 
\end{definition}
Notice that $h_w \in i\affinvol_n$ for every $w \in \widetilde {\mathcal{I}_n^0}$.
For example, if $n=5$ and 
\[w = (1\,2)(3\,4) (5)  \tau^{(2,2,-1,-1, -2)} \in \widetilde {\mathcal{I}_n^0},\]
 then 
\[ h_w =  \begin{pmatrix}
0 & it^2& 0 & 0 & 0 \\
it^2 & 0 & 0 & 0 & 0 \\
0 & 0 & 0 & it^{-1} & 0 \\
0 & 0 &  it^{-1}& 0 & 0 \\
0 & 0 & 0 & 0 & t^{-2}
\end{pmatrix}.
\]
However in case (ii) of Definition~\ref{def-aff-tw-invol} where matrices in $i\affinvol_n$ contain no non-zero diagonal entries, we need extra notations to index them properly by elements in $\widetilde {\mathcal{I}_n^0}$.
\begin{definition}\label{hw-plus-minus}
Suppose $w = \bar w \tau^{\mathbf{c}} \in \widetilde {\mathcal{I}_n^0}$. When $\bar w$ has no fixed points, we define $h_w^+ = h_w$ and $h_w^- = d h_w d$ where $d = \diag(-1, 1, 1, \dots 1)$.
\end{definition}
Notice that $h_w^+, h_w^- \in i\affinvol_n$.
For example if $n = 4$ and $w = (1\,2)(3\,4) \tau^{(2,2,-2,-2)}$, then
\[ h_w^+ =  \begin{pmatrix}
0 & it^2& 0 & 0 \\
 it^2 & 0 & 0 & 0 \\
0 & 0 & 0 & it^{-2} \\
0 & 0 &  it^{-2} & 0
\end{pmatrix} \quand 
h_w^- =  \begin{pmatrix}
0 & -it^2& 0 & 0 \\
 -it^2 & 0 & 0 & 0 \\
0 & 0 & 0 & it^{-2} \\
0 & 0 &  it^{-2} & 0
\end{pmatrix}
.\] 
We now recall some relevant identities.
For odd integers $a$ and $b$, we have
\begin{equation} \label{sl2diag22blockbreak}
\begin{pmatrix} t^a& 0 \\ 0 & t^b\end{pmatrix} =  \begin{pmatrix} t^{\frac{a-1}{2}}&t^{\frac{a-1}{2}} (t-1)^{\frac{1}{2}}\\  -(t-1)^{\frac{1}{2}}t^{\frac{b-1}{2}} & t^{\frac{b-1}{2}} \end{pmatrix}\begin{pmatrix} t^{\frac{a-1}{2}}&-(t-1)^{\frac{1}{2}}t^{\frac{b-1}{2}} \\ t^{\frac{a-1}{2}} (t-1)^{\frac{1}{2}} & t^{\frac{b-1}{2}} \end{pmatrix}. 
\end{equation}
Also for any $a\in \ZZ$,
\begin{equation} \label{sl2antidiag22blockbreak1}
  \begin{pmatrix} 0& it^a \\ it^a & 0\end{pmatrix} =  \begin{pmatrix} t^a/2 & it^a/2 \\ i & 1 \end{pmatrix}  \begin{pmatrix} t^a/2 & i  \\ it^a/2 & 1 \end{pmatrix}, 
  \end{equation}
\begin{equation} \label{sl2antidiag22blockbreak2}
  \begin{pmatrix} 0& -it^a \\ -it^a & 0\end{pmatrix} = \begin{pmatrix} i & 1 \\ -t^a/2 & -it^a/2 \end{pmatrix}  \begin{pmatrix} i & -t^a/2  \\ 1 & -it^a/2 \end{pmatrix}.  
  \end{equation}
\begin{definition}\label{def-gw-sortho}
Given $w  = \bar w \tau^{(c_1, c_2, \dots, c_n)} \in \widetilde {\mathcal{I}_n^0}$, define matrices $g_w$, $g_{w}^+$, and $g_{w}^-\in \SL_n(F)$ depending on whether $\bar w$ contains fixed points or not.
\ben
\item[(1)] If $\bar w$ contains one or more fixed points, we define $g_w \in \SL_n(F)$ as follows:
\ben
\item[$\bullet$]Replace each $t^{c_j}$ in $w$ with $t^{c_j/2}$ if $c_j$ is even and $\bar w(j) = j$.
\item[$\bullet$]There are even number of indices $j$'s satisfying $\bar w(j) = j$ and $c_j$'s being odd integers. If these indices are listed as $j_1 < j_2 < \dots < j_{2\ell}$ for some $\ell \in \mathbb{Z}_{>0}$, pair them as follows: $j_1$ with $j_2$, $j_3$ with $j_4$, and so on, up to $j_{2\ell-1}$ with $j_{2\ell}$.
Replace each $2$-by-$2$ submatrix $\{j,k\} \times \{j,k\}$ in $w$, such that $j$ and $k$ are paired, with the matrix 
\[ \begin{pmatrix} t^{\frac{a-1}{2}}&-(t-1)^{\frac{1}{2}}t^{\frac{b-1}{2}} \\ t^{\frac{a-1}{2}} (t-1)^{\frac{1}{2}} & t^{\frac{b-1}{2}} \end{pmatrix}.\]
\item[$\bullet$]For each pair $c_j = c_k$ that corresponds to a cycle $(j\,k)$ in $\bar w$, replace the $2$-by-$2$ submatrices $\{j,k\} \times \{j,k\}$ in $w$ with the matrix $\begin{pmatrix} t^a/2 & i  \\ it^a/2 & 1 \end{pmatrix}$.
\een
\item[(2)] If $\bar w$ contains no fixed points (necessarily $n$ is even), we define $g_{w}^+\in \SL_n(F)$ in the same way as $g_w$ in (1), with only the third bullet point being used.
\item[(3)] Define $g_{w}^-\in \SL_n(F)$ in the same way as $g_{w}^+$, with the only difference being that for the cycle $(1\,k)$ in $\bar w$, the $2$-by-$2$ submatrix $\{1,k\} \times \{1,k\}$ in $w$ is replaced with the matrix $\begin{pmatrix} i & -t^a/2  \\ 1 & -it^a/2 \end{pmatrix}$.
\een
\end{definition}
We give two examples illustrating the above procedures. Suppose $n=5$, and let 
\[ w = (1 \, 2) \tau^{(2,2,-1,-1, -2)} \quand h_w= \begin{pmatrix}
0 & it^2& 0 & 0 & 0 \\
it^2 & 0 & 0 & 0 & 0 \\
0 & 0 & t^{-1} & 0 & 0 \\
0 & 0 & 0& t^{-1} & 0 \\
0 & 0 & 0 & 0 & t^{-2}
\end{pmatrix}.\]
 Then
\[ g_w = \begin{pmatrix}
t^2/2 & i& 0 & 0 & 0 \\
 it^2/2 & 1 & 0 & 0 & 0 \\
0 & 0 & t^{-1}&-(t-1)^{\frac{1}{2}}t^{-1} & 0 \\
0 & 0 & t^{-1} (t-1)^{\frac{1}{2}} & t^{-1} & 0 \\
0 & 0 & 0 & 0 & t^{-1}
\end{pmatrix}.\]
Notice that $g_w^Tg_w = h_w$.

Suppose $n = 4$ and let $w = (1\,2)(3\,4) \tau^{(2,2,-2,-2)}$. Then 
\[h_w^+ =  \begin{pmatrix}
0 & it^2& 0 & 0 \\
 it^2 & 0 & 0 & 0 \\
0 & 0 & 0 & it^{-2} \\
0 & 0 &  it^{-2} & 0
\end{pmatrix} , \, h_w^- =  \begin{pmatrix}
0 & -it^2& 0 & 0 \\
 -it^2 & 0 & 0 & 0 \\
0 & 0 & 0 & it^{-2} \\
0 & 0 &  it^{-2} & 0
\end{pmatrix},\] and
\[g_{w}^+ = \begin{pmatrix}
t^2/2 & i& 0 & 0 \\
 it^2/2 & 1 & 0 & 0 \\
0 & 0 & t^{-2}/2 & i \\
0 & 0 &  it^{-2}/2 & 1
\end{pmatrix}, \quad 
g_{w}^- = \begin{pmatrix}
i & -t^2/2& 0 & 0 \\
1 & -it^2/2 & 0 & 0 \\
0 & 0 & t^{-2}/2 & i \\
0 & 0 &  it^{-2}/2 & 1
\end{pmatrix}.\]
Notice that $(g_{w}^+)^T g_{w}^+ = h_w^+$ and $(g_{w}^-)^T g_{w}^- = h_w^-$.

Again the following lemma holds by identities \eqref{sl2diag22blockbreak}, \eqref{sl2antidiag22blockbreak1} and \eqref{sl2antidiag22blockbreak2}.

\begin{lemma}\label{sqrt-sortho-affineperm}
If $w \in \widetilde {\mathcal{I}_n^0}$ is not fixed-point-free then $g_w^T g_w = h_w$. If $w \in \widetilde {\mathcal{I}_n^0}$ is fixed-point-free then $(g_{w}^+)^T g_{w}^+ = h_w^+$ and $(g_{w}^-)^T g_{w}^- = h_w^-$.
\end{lemma}
The following lemma explains why $Kg_{w}^+B$ and $Kg_{w}^-B$ are distinct double cosets if $\bar w$ is fixed-point-free. In this lemma we denote $B(\GL_n(F))$ and $B(\SL_n(F))$ as the Iwahori subgroups of $\GL_n(F)$ and $\SL_n(F)$ respectively. Recall that these consist of the matrices that are upper-triangular modulo $t$ in $\GL_n(A)$ and $\SL_n(A)$ respectively.

\begin{lemma} \label{fpf-aff-invol-two-orbits}
Suppose $b^Th_w^+b = h_w^+$, where $b \in B(\GL_n(F))$ and $w \in \widetilde {\mathcal{I}_n^0}$ are such that $h_w^+$ has no non-zero entries on the diagonal. Then $\det(b) = 1$, i.e., $b \in B(\SL_n(F))$.

\begin{proof}
Write $w = \bar w \tau^{(c_1, c_2, \dots, c_n)}$ as before, where $\bar w \in S_n$.
Since $\bar w$ is fixed-point-free in $S_n$, we can write $\bar w$ as $(a_1,b_1)(a_2,b_2)\dots (a_{n/2},b_{n/2})$, where $\{a_1, a_2, \dots ,a_{n/2}, b_1, b_2, \dots, b_{n/2}\} = [1,n]$, and $a_j < b_j$ for $j \in [1,n/2]$. Denote the diagonal entries of $b$ as $d_{11}, d_{22}, \dots, d_{nn} \in A^*$, and their respective non-zero constant terms as $c_{11}, c_{22}, \dots, c_{nn} \in \K^*$.

Notice that the equation $b^Th_w^+b = h_w^+$ can be rewritten as $((h_w^+)^{-1}b^Th_w^+)b = 1$. Since $h_w^+$ is a symmetric monomial matrix, left and right multiplying $(h_w^+)^{-1}$ and $h_w^+$ to a matrix permutes rows and columns of that matrix respectively, with multiplication of integral powers of $t$ to rows and columns. Since $w^{-1} = w^*$, for $j \in [1,n/2]$, left multiplying $(h_w^+)^{-1}$ to $b^T$ moves the entry $d_{a_ja_j}$ and $d_{b_jb_j}$ to the $(b_j, a_j)$ and $(a_j, b_j)$ entries respectively, with the multiplication of the same integral power of $t$ times $-i$. 

Next, right multiplying $h_w^+$ to $(h_w^+)^{-1}b^T$ moves the entry $d_{a_ja_j}$ and $d_{b_jb_j}$ to the $(b_j, b_j)$ and $(a_j, a_j)$ entries respectively, canceling the integer power of $t$ and the $-i$ multiples in the $(b_j, a_j)$ and $(a_j, b_j)$ entries of $(h_w^+)^{-1}b^T$. Since $b$ and $(h_w^+)^{-1}b^Th_w^+$ are both upper triangular modulo $t$, the constant term of the $(a_j,a_j)$ and $(b_j,b_j)$ entries in $(w^{-1}b^Tw)b$ is $c_{a_ja_j}c_{b_jb_j}$. Since $((h_w^+)^{-1}b^Th_w^+)b = 1$, $c_{a_ja_j}c_{b_jb_j} = 1$ for $j \in [1,n/2]$.

Since $\det(h_w^+) = \det(b^Th_w^+b) = \det(b)^2\det(h_w^+)$, it holds that $\det(b) = 1$ or $-1$. Again, since $b \in B(\GL_n(F))$, $\det(b)$ has a constant term equal to the product of the constant terms of the diagonal entries of $b$. Since we have concluded in the previous paragraph that $c_{a_ja_j}c_{b_jb_j} = 1$ for $j \in [1,n/2]$, $\det(b)$ has a constant term equal to $1$, and hence $\det(b) = 1$.
\end{proof}
\end{lemma}

The following is the proof of Theorem~\ref{sortho-orbit-thm-intro}, concerning the double coset representatives in $\SO_n(F) \backslash \SL_n(F) / B$. Recall that $G = \SL_n(F)$ and $K = \SO_n(F)$. 

A equivalent statement of Theorem~\ref{sortho-orbit-thm-intro} is the following decomposition:
\[ G = \bigsqcup_{\substack{w \in \widetilde {\mathcal{I}_n^0}\\ \bar w \text{ not fpf}}} Kg_w B \sqcup \bigsqcup_{\substack{w \in \widetilde {\mathcal{I}_n^0}\\ \bar w \text{ fpf}}} (Kg_{w}^+B \sqcup Kg_{w}^-B),\]
where ``fpf" is the abbreviation of ``fixed-point-free".
In other words, a set of distinct double coset representatives of $K\backslash G/B$ can be constructed in the following way: for each $w  = \bar w \tau^{(c_1, c_2, \dots, c_n)} \in \widetilde {\mathcal{I}_n^0}$,
include $g_w$ if $\bar w$ contains a fixed point, and otherwise include both $g_{w}^+$ and $g_{w}^-$.

\begin{proof}[Proof of Theorem~\ref{sortho-orbit-thm-intro}]
As in the proof of Theorem~\ref{ortho-orbit-thm-intro}, there is a bijection between $K\backslash G$ and $\{g^Tg: g \in \SL_n(F)\}$. Under this bijection, the orbits for the right $B$-action $b: h \mapsto b^T hb$ on the set $\{g^Tg: g \in \SL_n(F)\}$ correspond to right $B$-orbits on $K\backslash G$. By a suitable series of right $B$-actions, we can transform any matrix of the form $g^Tg$ with a symmetric monomial matrix of determinant $1$. The inductive process to accomplish this is the same as in the proof of Theorem~\ref{ortho-orbit-thm-intro}, 
with the difference that we perform calculations for the $2$-by-$2$ blocks using the tools provided by Proposition~\ref{sorthogonaln=2-gl2}.

Now we know that each $B$-orbit of $\{g^Tg: g \in \SL_n(F)\}$ contains a symmetric monomial matrix with the determinant equal to $1$, denoted by $h$. Through another series of right $B$-actions with diagonal matrices, we can reduce each non-zero entry on the diagonal to an integral power of $t$ and each non-diagonal non-zero entry into $i$ or $-i$ times an integral power of $t$. This process can be outlined as follows: suppose the submatrix $\begin{psmallmatrix} 0 & t^m c \\ t^m c & 0\end{psmallmatrix}$, where $m \in \ZZ$ and $1 \neq c \in A^*$, appears in the $j$-th and $k$-th rows and columns of $h$. In this case, proceed with one of the following steps:
\ben
\item Find a non-zero diagonal entry, located in row and column $\ell \neq j, k$, which has not yet been reduced to an integral power of $t$. After the right $B$-action of a diagonal matrix in $\SL_n(F)$ that has $ic^{-1}$ in the $(j, j)$ entry, $-ic$ in the $(\ell, \ell)$ entry, and $1$ in all other diagonal entries, the $(j,k)$ and $(k,j)$ entries of the resulting matrix are equal to $it^m$. This reduction process takes the following form within the $j$-th, $k$-th, and $\ell$-th rows and columns:
\begin{equation} \label{so-adjust-entry1}
 \begin{pmatrix} ic^{-1} & 0 & 0 \\
0 & 1 & 0 \\
0 & 0 & -ic \end{pmatrix}
\begin{pmatrix} 0 & t^m c & 0 \\
t^m c & 0 & 0 \\
0 & 0 & d \end{pmatrix}
\begin{pmatrix} ic^{-1} & 0 & 0 \\
0 & 1 & 0 \\
0 & 0 & -ic \end{pmatrix} = 
\begin{pmatrix} 0 & i t^m  & 0 \\
i t^m & 0 & 0 \\
0 & 0 & dc^2 \end{pmatrix},
\end{equation}
where $d \in F$.
\item Find another pair of non-zero non-diagonal entries, such as the $(\ell_1, \ell_2)$ and $(\ell_2, \ell_1)$ entries, which have not yet been reduced to an integral power of $t$. After the right $B$-action of a diagonal matrix in $\SL_n(F)$ with $ic^{-1}$ in the $(j, j)$ entry, $-ic$ in the $(\ell_1, \ell_1)$ entry, and $1$ in all other diagonal entries, the $(j,k)$ and $(k,j)$ entries of the resulting matrix are equal to $it^m$. This reduction process takes the following form within the $j$-th, $k$-th, $\ell_1$-th, and $\ell_2$-th rows and columns:
\begin{equation}\label{so-adjust-entry2}
 \begin{pmatrix} ic^{-1} & 0 & 0 & 0  \\
0 & 1 & 0 & 0 \\
0 & 0 & -ic & 0\\
0 & 0 & 0 & 1 \end{pmatrix}
\begin{pmatrix} 0 & t^m c & 0 & 0  \\
t^m c & 0 & 0 & 0 \\
0 & 0 & 0 & d\\
0 & 0 & d & 0 \end{pmatrix}
\begin{pmatrix} ic^{-1} & 0 & 0 & 0  \\
0 & 1 & 0 & 0 \\
0 & 0 & -ic & 0\\
0 & 0 & 0 & 1 \end{pmatrix} = 
\begin{pmatrix} 0 & i t^m  & 0 & 0 \\
i t^m & 0 & 0  & 0\\
0 & 0 & 0 & -icd \\
0 & 0 & -icd & 0\end{pmatrix}, 
\end{equation}
where $d \in F$.
\een
Suppose an entry $t^m c$, where $m \in \ZZ$ and $1 \neq c \in A^*$, appears in the $(j, j)$ entry of $h$. In this case, locate another non-zero diagonal entry, situated in row and column $\ell \neq j$, which has not yet been reduced to an integral power of $t$ (if it exists). After the right $B$-action of a diagonal matrix in $\SL_n(F)$ with $c^{-1/2}$ in the $(j, j)$ entry, $c^{1/2}$ in the $(\ell, \ell)$ entry, and 1 in all other diagonal entries, the $(j,j)$ entry of the resulting matrix is equal to $t^m$. This reduction process is applied to the $j$-th and $\ell$-th rows and columns in the following manner:
\begin{equation}\label{so-adjust-entry3}
\begin{pmatrix} c^{-1/2} & 0 \\
 0 & c^{1/2} \end{pmatrix}
\begin{pmatrix} t^m c & 0\\
 0 & d \end{pmatrix}
\begin{pmatrix} c^{-1/2} & 0 \\
 0 & c^{1/2} \end{pmatrix} = 
\begin{pmatrix} t^m  & 0\\
 0 & dc \end{pmatrix}, \text{ where } d \in F.
 \end{equation}

Suppose $h$ contains at least one non-zero diagonal entry. Begin by reducing the non-diagonal non-zero entries using \eqref{so-adjust-entry1} and \eqref{so-adjust-entry2}, followed by reductions on the non-zero diagonal entries using \eqref{so-adjust-entry3}. The determinant of the resulting matrix equals $1$, and each $2$-by-$2$ block of the form $\begin{psmallmatrix} 0 & it^m \ it^m & 0\end{psmallmatrix}$ has a determinant of $t^{2m}$. Consequently, the remaining entry on the diagonal must be some integral power of $t$, and the resulting matrix must align with the form $h_w$ as defined in Definition~\ref{hw}.

Alternatively, if $h$ contains no non-zero diagonal entry, reduce the non-diagonal non-zero entries not located in the first row and column using \eqref{so-adjust-entry1} and \eqref{so-adjust-entry2}. In the resulting matrix, suppose the $(1, j)$ and $(j, 1)$ entries are non-zero for some $j$ and equal to $t^m c$ for some $c \in A^*$. Considering that the determinant remains $1$ and each $2$-by-$2$ block of the form $\begin{psmallmatrix} 0 & it^m \\ it^m & 0\end{psmallmatrix}$ yields a determinant of $t^{2m}$, it follows that $-c^2 = 1$, leading to $c = i$ or $-i$. Consequently, the resulting matrix is either in the form $h_w^+$ or $h_w^-$ as outlined in Definition~\ref{hw-plus-minus}.

As per Lemma~\ref{sqrt-sortho-affineperm}, each double coset in $K \backslash G/B$ includes at least one of the matrices $g_w, g_w^+$, or $g_w^-$. Through arguments similar to the ones presented in the proof of Theorem~\ref{ortho-orbit-thm-intro}, it becomes evident that $g_w$ is not in the same double coset as $g_{v}^+$ or $g_{v}^-$ when $v \neq w$.

The only fact that remains to be checked is that for $w \in \widetilde {\mathcal{I}_n^0}$ such that $\bar w$ is fixed-point-free in $S_n$, the $g_{w}^+$ and $g_{w}^-$ do not lie in the same double coset. 

Suppose there exists $b \in B \subset \SL_n(A)$ such that 
\[b^T (g_{w}^+)^T g_{w}^+ b = b^T h_w^+ b = h_w^- = (g_{w}^-)^T g_{w}^-.\]
Notice that
\[ h_w^-= dh_w^+d,\]
where $d = \diag(-1, 1,1, \dots, 1)$. It follows that \[(bd)^T h_w^+ bd = h_w^+.\]
 By Lemma~\ref{fpf-aff-invol-two-orbits}, $\det(bd) = 1$. However, $\det(d) = -1$. Therefore, $\det(b) = -1$, contradicting the assumption that $b \in B \subset \SL_n(A)$.

Therefore, we conclude that $g_{w}^+$ and $g_{w}^-$ do not lie in the same double coset.
\end{proof}

\section{Orbits of the symplectic group}\label{sect4-ref}
In this section, we focus on the orbits of the symplectic group over the affine flag variety.
 
Define 
\[ J = \begin{pmatrix} 0 &1_n \\ -1_n & 0 \end{pmatrix}.\] 
Here, $1_n$ refers to the $n$-by-$n$ identity matrix. Note that $J^T = J^{-1} = -J$. Notice also we are considering $(2n)$-by-$(2n)$ matrices now since symplectic groups are only defined in even dimensions.

Now we consider the case where $G = \GL_{2n}(F)$ with 
\[K = \Sp_{2n}(F) = \{g \in \GL_{2n}(F) : g^TJg = J\}.\]
The following lemma is well-known.
\begin{lemma}
The set  $\{g^TJg : g \in G\}$ is equal to the set $\{ h \in G: h = -h^T\}$.
\end{lemma}

We now give the necessary definition for the combinatorial indexing set for the double cosets in $K\backslash G/ B$. Recall from Definition~\ref{def-fpf-ex-aff-tw-invol} that
the set $\skewsym_{2n}$ consists of all skew-symmetric $2n$-by-$2n$ monomial matrices whose non-zero entries above the diagonal are integral powers of $t$.
\begin{remark}
Define the set of \defn{fixed-point-free extended affine twisted involutions} $\IIfpf$ to be the subset of $\widetilde{S}_{2n}^+$ consisting of symmetric affine permutation matrices with no non-zero diagonal entries.
Every matrix $h \in \skewsym_{2n}$ can be indexed by $w \in \IIfpf$, as stated in the following definition.
\end{remark}

\begin{definition}
For any $w \in \IIfpf$, define the skew-symmetric matrix $h^{\mathsf{sk}}_w= -(h^{\mathsf{sk}}_w)^T$ by assigning a negative sign to all the non-zero entries in the lower triangular part of $w$. 
\end{definition}
It holds that $h^{\mathsf{sk}}_w \in \skewsym_{2n}$ for any $w \in \IIfpf$. As an example suppose $w = (1\,2)(3\,4) \tau^{(2,2,-2,-2)} \in \IIfpf$. Then 
\[h^{\mathsf{sk}}_w =  \begin{pmatrix}
0 & t^2& 0 & 0 \\
 -t^2 & 0 & 0 & 0 \\
0 & 0 & 0 & t^{-2} \\
0 & 0 &  -t^{-2} & 0
\end{pmatrix}.
\]
The following definition concerns the double coset representatives in $ K \backslash G/B$. Notice that $J = h^{\mathsf{sk}}_{w_J} \in \skewsym_{2n}$, where $w_J = (1\,\, n+1) (2 \,\, n+2) \dots (n \,\, 2n) \in S_{2n}$. Also recall that any fixed-point-free involution $w$ in $S_{2n}$ has the same cycle-type as $w_J$, so $w$ must be in the same conjugacy class as $w_J$. We specify a particular $\sigma_w \in S_{2n}$ such that $\sigma_w^{-1} w_J \sigma_w = w$ as follows:
\ben
\item Write $w$ as a product of $n$ disjoint $2$-cycles. Set $\sigma_0 = \text{id} \in S_{2n}$, and $w_0 = w$.
\item Inductively for $j \in [1,n]$, if $w_{j-1}(j) = j+n$ then set $\sigma_{j} = \sigma_{j-1}$ and $w_j = w_{j-1}$, and otherwise set $\sigma_{j} = \sigma_{j-1} (j+n \, \, \, w_{j-1}(j))$, and 
$w_j = (j+n \, \, \, w_{j-1}(j)) w_{j-1} (j+n \, \,\, w_{j-1}(j)).$
\item Finally define $\sigma_w$  to be the permutation $\sigma_n$. 
\een
Because $\sigma_w^{-1} w_J \sigma_w = w$, 
it follows that the skew-symmetric matrix $\sigma_w^{T} J \sigma_w$ has the same non-zero positions as the (symmetric) permutation matrix of $w$.

For the following definition let $w = \bar w \tau^{\mathbf{c}} \in \IIfpf$, and write $\bar w = (i_1\,j_1)(i_2\,j_2) \dots (i_n\,j_n)$ as a product of $n$ disjoint $2$-cycles with $i_1 < i_2 < \dots < i_n$ and $i_k < j_k$ for each $k \in [1,n]$.
\begin{definition}\label{def-gw-sp}
Define $\sigma_w \in S_{2n}$ as above.
Form $\sqrt{\mathbf{c}} \in \ZZ^n$ from $\mathbf{c}$ by changing the entries $c_{j_k}$  to $0$ for all $k \in [1,n]$. Let $s \in S_{2n}$ be the product of the cycles $(i_k \, j_k)$ for which the $(i_k, j_k)$-entry of $\sigma_w^T J \sigma_w$ is equal to $-1$. Finally define $g_w \in \GL_{2n}(F)$ to be the affine permutation matrix of  $\sigma_w \tau^{\sqrt{\mathbf{c}}}s$.
\end{definition}
As an example, suppose $w = (1\,2)(3\,4) (5\,6) \tau^{(1,1,-3,-3,2, 2)}$. Notice that $(3\,5)(2\, 4) w_J (2\,4)(3\, 5) = (1\,2)(3\,4) (5\,6) = \bar w$. Therefore $\sigma_w = (2\,4)(3\, 5)$. Now $\sqrt{(1,1,-3,2,-3, 2)} = (1,0,-3,2,0,0)$, and 
\[ \sigma_w^T J \sigma_w = \begin{pmatrix} 
0& 1 & 0 & 0 & 0 & 0\\
-1 & 0 & 0 & 0 & 0 & 0\\
0 & 0 & 0&-1 & 0 & 0\\
0 & 0 & 1& 0 & 0 & 0\\
0 & 0 & 0 & 0 & 0 & 1\\
0 & 0 & 0 & 0 & -1 & 0
\end{pmatrix}.\]
Therefore $s = (3\,4)$ and hence 
\begin{align*}
 g_w &= \begin{pmatrix} 
1& 0 & 0 & 0 & 0 & 0\\
0 & 0 & 0 & 1& 0 & 0\\
0 & 0 & 0&0 & 1 & 0\\
0 & 1& 0& 0 & 0 & 0\\
0 & 0 & 1 & 0 & 0 &0 \\
0 & 0 & 0 & 0 & 0 & 1
\end{pmatrix}
 \begin{pmatrix} 
t& 0 & 0 & 0 & 0 & 0\\
0 & 1 & 0 & 0& 0 & 0\\
0 & 0 & t^{-3}&0 & 0 & 0\\
0 & 0& 0& 1 & 0 & 0\\
0 & 0 & 0 & 0 & t^2 &0 \\
0 & 0 & 0 & 0 & 0 & 1
\end{pmatrix}
 \begin{pmatrix} 
1& 0 & 0 & 0 & 0 & 0\\
0 & 1 & 0 & 0& 0 & 0\\
0 & 0 & 0&1 & 0 & 0\\
0 & 0& 1& 0 & 0 & 0\\
0 & 0 & 0 & 0 & 1 &0 \\
0 & 0 & 0 & 0 & 0 & 1
\end{pmatrix}\\
&= \begin{pmatrix} 
t& 0 & 0 & 0 & 0 & 0\\
0 & 0 & 1 & 0& 0 & 0\\
0 & 0 & 0&0 & t^2 & 0\\
0 & 1& 0& 0 & 0 & 0\\
0 & 0 & 0 & t^{-3} & 0 &0 \\
0 & 0 & 0 & 0 & 0 & 1
\end{pmatrix}.
\end{align*}
Notice that $g_w^T J g_w = \begin{pmatrix} 
0& t & 0 & 0 & 0 & 0\\
-t & 0 & 0 & 0 & 0 & 0\\
0 & 0 & 0&t^{-3} & 0 & 0\\
0 & 0 & -t^{-3}& 0 & 0 & 0\\
0 & 0 & 0 & 0 & 0 & t^2\\
0 & 0 & 0 & 0 & -t^2 & 0
\end{pmatrix} =  h^{\mathsf{sk}}_{w}$.

A result similar to Lemma~\ref{sqrt-ortho-affineperm} and \ref{sqrt-sortho-affineperm} holds for this $g_w$.
\begin{lemma}\label{sqrt-sp-affineperm}
It holds that $g_w^TJg_w =  h^{\mathsf{sk}}_w$ for all $w \in  \IIfpf$.
\end{lemma}
For $w \in \IIfpf$, define $\mathcal{O}_w = \Sp_{2n}(F) g_w B = Kg_wB$.
Recall that Theorem~\ref{sp-orbit-thm-intro} is equivalent to stating that
the map 
$h_w^{\mathsf{sk}} \mapsto \mathcal{O}_w $
is a bijection between $\skewsym_{2n}$ and  $ K \backslash G/B$.

\begin{proof}[Proof of Theorem~\ref{sp-orbit-thm-intro}]
The proof is similar to the $\O_n$-case.

The quotient $K\backslash G$ is in bijection with the set $ \{g^TJg : g \in G\}$ via the map $Kg \mapsto g^TJg$. The set $ \{g^TJg : g \in G\}$ contains all the skew-symmetric matrices in $G$, and is closed under the right $B_{2n}$-action $(g^TJg)\cdot b := b^Tg^TJgb$, which commutes with the right $B$-action $Kg\cdot b := Kgb$ via the bijection.

Given $h = g^TJg$ where $g \in G$, we consider the non-zero upper-triangular entries in $h$ with the least order. Among these least order entries, let $(i,j)$ be the position of the one that is leftmost and uppermost along its column and row. Denote the $(i,j)$-entry of $h$ by $h_{ij}$.
In other words, all the entries above $h_{ij}$ have strictly greater order, and all the entries below have an order that is not less than that of $h_{ij}$.

It always holds that $i \neq j$ since $h$ is skew-symmetric and $h_{ii} = 0$ for $1\leq i \leq n$. Without loss of generality, assume $i > j$ since $h_{ij} = h_{ji}$. By a similar right action of a diagonal matrix, we can reduce $h_{ij}$ and $h_{ji}$ to be powers of $t$. Applying the right actions of elementary matrices as described in the proof of Theorem~\ref{ortho-orbit-thm-intro}, we can eliminate all the entries in the $i$-th and $j$-th columns and rows except $h_{ij} = h_{ji}$. 
We can further conjugate by a diagonal matrix so that the resulting matrix contains $t^{\ord(h_{ij})}$ as the only non-zero entry in the $i$-th and $j$-th columns.

By applying the above procedure $n$ times, we can conclude that there exists $b \in B$ such that $b^Thb$ is a fixed-point-free extended affine twisted involution, with equal powers of $t$ reflected along the main diagonal. Write $h = g^TJg$ for some $g\in G$, and denote the fixed-point-free extended affine twisted involution $b^Thb$ as $w_g$, such that $h = h^{\mathsf{sk}}_{w_g} \in \skewsym_{2n}$.  

The above construction proves the existence of $w_g$ for $K g B$ with $g \in G$, and the construction of $g_w$ given before the theorem provides the desired property of $g_w$ for a given fixed-point-free extended affine twisted involution. Thus, we have proven that the map $w \mapsto \mathcal{O}_w$ is a surjection.

It remains to show that the map is an injection. Suppose $\mathcal{O}_w = \mathcal{O}_{w'}$ for some $w, w' \in \IIfpf$. A argument similar to  the proof of Theorem~\ref{ortho-orbit-thm-intro} applies. We observe that 
\[ h^{\mathsf{sk}}_{w'} \in \mathcal{O}_{w'} = \mathcal{O}_{w} = \{b^Th_wb : b \in B\}.\] 
Therefore, we have $h^{\mathsf{sk}}_{w'} = b^Th^{\mathsf{sk}}_wb$ for some $b \in B$. By suitable multiplications of diagonal matrices on the left of $h^{\mathsf{sk}}_{w'}$ and $h^{\mathsf{sk}}_w$ respectively, we have $w' = b'wb$, where $b' \in B'$, the opposite Iwahori subgroup. 
Similar to Lemma~\ref{uniq-orbit-ortho}, we have $w' = w$ and hence $h^{\mathsf{sk}}_w = h^{\mathsf{sk}}_{w'}$ by the affine Bruhat decomposition \cite[Theorem 2.16]{IM}. 
Therefore, elements $h^{\mathsf{sk}}_w$ in $\skewsym_{2n}$ give rise to distinct double cosets $K g_w B$.
\end{proof}

\addcontentsline{toc}{section}{Bibliography}
\printbibliography

\end{document}

%% file: preamble.tex
\usepackage{latexsym}
\usepackage{amssymb}
\usepackage{amsthm}
\usepackage{amscd}
\usepackage{amsmath}
\usepackage{mathrsfs}
\usepackage{graphicx}
\usepackage{shuffle}
\usepackage[colorlinks=true, pdfstartview=FitV, linkcolor=blue, citecolor=blue, urlcolor=blue]{hyperref}
\usepackage{mathtools}
\usepackage[bbgreekl]{mathbbol}
\usepackage{mathdots}
\usepackage{ytableau}
\usepackage[enableskew,vcentermath]{youngtab}
\usepackage{tikz-cd}

\usepackage{mathtools}
\usepackage{amssymb}

\usepackage{collcell}
\usepackage{datatool}

\usepackage[colorinlistoftodos]{todonotes}
\usetikzlibrary{chains,scopes}
\usetikzlibrary{shapes.geometric,positioning}

\DeclareSymbolFontAlphabet{\mathbb}{AMSb}
\DeclareSymbolFontAlphabet{\mathbbl}{bbold}

\def\Sp{\mathsf{Sp}}
\def\O{\mathsf{O}}
\def\SO{\mathsf{SO}}
\def\K{\mathbb{K}}

\definecolor{darkred}{rgb}{0.7,0,0} 
\newcommand{\defn}[1]{{\color{darkred}\emph{#1}}} 

\numberwithin{equation}{section}
\allowdisplaybreaks[1]

\theoremstyle{definition}
\newtheorem* {theorem*}{Theorem}
\newtheorem* {proposition*}{Proposition}
\newtheorem* {corollary*}{Corollary}
\newtheorem* {conjecture*}{Conjecture}
\newtheorem{theorem}{Theorem}[section]

\theoremstyle{definition}

\newtheorem* {example*}{Example}

\newtheorem{lemma}[theorem]{Lemma}
\theoremstyle{definition}
\newtheorem{definition}[theorem]{Definition}
\theoremstyle{definition}

\newtheorem{proposition}[theorem]{Proposition}
\newtheorem{corollary}[theorem]{Corollary}

\newtheorem{remark}[theorem]{Remark}
\newtheorem*{remark*}{Remark}
\theoremstyle{definition}

\theoremstyle{definition}

\theoremstyle{definition}

\theoremstyle{definition}

\def\({\left(}
\def\){\right)}

\newcommand{\CC}{\mathbb{C}}

\def\Hom{\mathrm{Hom}}

\def\CC{\mathbb{C}}

\def\ZZ{\mathbb{Z}}

\def\GL{\textsf{GL}}
\def\SL{\textsf{SL}}

\def\barr{\begin{array}}
\def\earr{\end{array}}
\def\ba{\begin{aligned}}
\def\ea{\end{aligned}}
\def\be{\begin{equation}}
\def\ee{\end{equation}}

\def\quand{\quad\text{and}\quad}

\def\ben{\begin{enumerate}}
\def\een{\end{enumerate}}

\def\bei{\begin{itemize}}
\def\eei{\end{itemize}}

\def\Perm{\textsf{Perm}}

\def\path{\mathsf{path}}

\def\whSym
{\mathfrak{m}\textsf{Sym}}

\def\whQSym
{\mathfrak{m}\textsf{QSym}}

\def\NN{\mathbb{N}}

\def\bar{\overline}

\def\diag{\textsf{diag}}

\usepackage{listings}
\lstdefinelanguage{Sage}[]{Python}
{morekeywords={False,sage,True},sensitive=true}
\definecolor{dblackcolor}{rgb}{0.0,0.0,0.0}
\definecolor{dbluecolor}{rgb}{0.01,0.02,0.7}
\definecolor{dgreencolor}{rgb}{0.2,0.4,0.0}
\definecolor{dgraycolor}{rgb}{0.30,0.3,0.30}

\usepackage{bbm}

\def\CCCt{\CC(\hspace{-0.5mm}(t)\hspace{-0.5mm})}

\def\KLt{\mathbb{K}(\hspace{-0.5mm}(t)\hspace{-0.5mm})}
\def\KPt{\mathbb{K}[\hspace{-0.5mm}[t]\hspace{-0.5mm}]}

\newcommand{\ord}{\operatorname{ord}}

\def\affinvol{\textsf{SymAPM}}
\def\eaffinvol{\textsf{eSymAPM}}

\def\IIfpf{ \widetilde{\mathcal{I}}_{2n}^{\text{fpf}}}

\def\skewsym{\textsf{SkewAPM}}

%% file: references.bib
@article {CJW,
    AUTHOR = {Can, Mahir Bilen and Joyce, Michael and Wyser, Benjamin},
     TITLE = {Chains in weak order posets associated to involutions},
   JOURNAL = {J. Combin. Theory Ser. A},
  FJOURNAL = {Journal of Combinatorial Theory. Series A},
    VOLUME = {137},
      YEAR = {2016},
     PAGES = {207--225},
      ISSN = {0097-3165,1096-0899},
   MRCLASS = {14M27 (05E18 06A07)},
  MRNUMBER = {3403521},
MRREVIEWER = {Vikraman\ Uma},
       URL = {https://doi.org/10.1016/j.jcta.2015.09.001},
}

@article {IM,
    AUTHOR = {Iwahori, Nagayoshi and Matsumoto, Hideya},
     TITLE = {On some {B}ruhat decomposition and the structure of the
              {H}ecke rings of $\mathfrak{p}$-adic {C}hevalley groups},
   JOURNAL = {Inst. Hautes \'Etudes Sci. Publ. Math.},
  FJOURNAL = {Institut des Hautes \'Etudes Scientifiques. Publications
              Math\'ematiques},
    NUMBER = {25},
      YEAR = {1965},
     PAGES = {5--48},
      ISSN = {0073-8301,1618-1913},
   MRCLASS = {20.70 (14.50)},
  MRNUMBER = {185016},
MRREVIEWER = {Rimhak\ Ree},
       URL = {http://www.numdam.org/item?id=PMIHES_1965__25__5_0},
}

@book {Kumar,
    AUTHOR = {Kumar, Shrawan},
     TITLE = {Kac-{M}oody groups, their flag varieties and representation
              theory},
    SERIES = {Progress in Mathematics},
    VOLUME = {204},
 PUBLISHER = {Birkh\"auser Boston Inc., Boston, MA},
      YEAR = {2002},
     PAGES = {xvi+606},
      ISBN = {0-8176-4227-7},
   MRCLASS = {22E46 (14M15 17B67 22E65)},
  MRNUMBER = {1923198},
MRREVIEWER = {Guy\ Rousseau},
       URL = {https://doi.org/10.1007/978-1-4612-0105-2},
}

@misc{Lusztig,
      title={Comments on my papers}, 
      author={Lusztig, George},
      year={2021},
      eprint={1707.09368},
      archivePrefix={arXiv},
      primaryClass={math.RT},
      url={https://arxiv.org/abs/1707.09368}, 
}

@article{Magyar,
  title={Affine Schubert varieties and circular complexes},
  author={Magyar, Peter},
  journal={arXiv:0210151},
  year={2002}
}

@phdthesis{Mann,
  title={Geometric Satake isomorphism for real reductive groups},
  author={Mann, Elizabeth},
  year={2003},
  school={University of Oxford}
}

@incollection {MO,
    AUTHOR = {Matsuki, Toshihiko and \=Oshima, Toshio},
     TITLE = {Embeddings of discrete series into principal series},
 BOOKTITLE = {The orbit method in representation theory ({C}openhagen,
              1988)},
    SERIES = {Progr. Math.},
    VOLUME = {82},
     PAGES = {147--175},
 PUBLISHER = {Birkh\"auser Boston, Boston, MA},
      YEAR = {1990},
      ISBN = {0-8176-3474-6},
   MRCLASS = {22E46},
  MRNUMBER = {1095345},
MRREVIEWER = {Brian\ D.\ Boe},
}

@article {Nadler04,
    AUTHOR = {Nadler, David},
     TITLE = {Matsuki correspondence for the affine {G}rassmannian},
   JOURNAL = {Duke Math. J.},
  FJOURNAL = {Duke Mathematical Journal},
    VOLUME = {124},
      YEAR = {2004},
    NUMBER = {3},
     PAGES = {421--457},
      ISSN = {0012-7094,1547-7398},
   MRCLASS = {22E67 (14M15)},
  MRNUMBER = {2084612},
MRREVIEWER = {Ulrich\ G\"ortz},
       URL = {https://doi.org/10.1215/S0012-7094-04-12431-5},
}

@article {Richardson90,
    AUTHOR = {Richardson, Roger Wolcott and Springer, Tony Albert},
     TITLE = {The {B}ruhat order on symmetric varieties},
   JOURNAL = {Geom. Dedicata},
  FJOURNAL = {Geometriae Dedicata},
    VOLUME = {35},
      YEAR = {1990},
    NUMBER = {1-3},
     PAGES = {389--436},
      ISSN = {0046-5755,1572-9168},
   MRCLASS = {20G15 (20G20)},
  MRNUMBER = {1066573},
MRREVIEWER = {Aloysius\ Helminck},
       URL = {https://doi.org/10.1007/BF00147354},
}

@book {PS,
    AUTHOR = {Pressley, Andrew and Segal, Graeme},
     TITLE = {Loop groups},
    SERIES = {Oxford Mathematical Monographs},
      NOTE = {Oxford Science Publications},
 PUBLISHER = {The Clarendon Press, Oxford University Press, New York},
      YEAR = {1986},
     PAGES = {viii+318},
      ISBN = {0-19-853535-X},
   MRCLASS = {22E65 (58D15 81D15)},
  MRNUMBER = {900587},
MRREVIEWER = {Jouko\ Mickelsson},
}

@article {Sambale,
    AUTHOR = {Sambale, Benjamin},
     TITLE = {An invitation to formal power series},
   JOURNAL = {Jahresber. Dtsch. Math.-Ver.},
  FJOURNAL = {Jahresbericht der Deutschen Mathematiker-Vereinigung},
    VOLUME = {125},
      YEAR = {2023},
    NUMBER = {1},
     PAGES = {3--69},
      ISSN = {0012-0456,1869-7135},
   MRCLASS = {13F25 (05A15 05A17 11P84 16W60)},
  MRNUMBER = {4552577},
MRREVIEWER = {Xiaolei\ Zhang},
       URL = {https://doi.org/10.1365/s13291-022-00256-6},
}

@incollection {Springer,
    AUTHOR = {Springer, Tonny Albert},
     TITLE = {Some results on algebraic groups with involutions},
 BOOKTITLE = {Algebraic groups and related topics ({K}yoto/{N}agoya, 1983)},
    SERIES = {Adv. Stud. Pure Math.},
    VOLUME = {6},
     PAGES = {525--543},
 PUBLISHER = {North-Holland, Amsterdam},
      YEAR = {1985},
      ISBN = {0-444-87711-8},
   MRCLASS = {20G05 (14M15)},
  MRNUMBER = {803346},
MRREVIEWER = {S.\ I.\ Gelfand},
       URL = {https://doi.org/10.2969/aspm/00610525},
}

@book {Vogan,
    AUTHOR = {Vogan, Jr., David A.},
     TITLE = {Representations of real reductive {L}ie groups},
    SERIES = {Progress in Mathematics},
    VOLUME = {15},
 PUBLISHER = {Birkh\"auser, Boston, MA},
      YEAR = {1981},
     PAGES = {xvii+754},
      ISBN = {3-7643-3037-6},
   MRCLASS = {22E47 (22E46)},
  MRNUMBER = {632407},
MRREVIEWER = {Joe\ Repka},
}

@article{Wyser,
  title={Symmetric subgroup orbit closures on flag varieties: Their equivariant geometry, combinatorics, and connections with degeneracy loci},
  author={Wyser, Benjamin J.},
  journal={arXiv:1201.4397},
  year={2012}
}

@article {Yamamoto,
    AUTHOR = {Yamamoto, Atsuko},
     TITLE = {Orbits in the flag variety and images of the moment map for
              classical groups. {I}},
   JOURNAL = {Represent. Theory},
  FJOURNAL = {Representation Theory. An Electronic Journal of the American
              Mathematical Society},
    VOLUME = {1},
      YEAR = {1997},
     PAGES = {329--404},
      ISSN = {1088-4165},
   MRCLASS = {22E46 (22E60)},
  MRNUMBER = {1479152},
MRREVIEWER = {William\ M.\ McGovern},
       URL = {https://doi.org/10.1090/S1088-4165-97-00007-1},
}

@book {Bump,
    AUTHOR = {Bump, Daniel},
     TITLE = {Lie groups},
    SERIES = {Graduate Texts in Mathematics},
    VOLUME = {225},
   EDITION = {Second},
 PUBLISHER = {Springer, New York},
      YEAR = {2013},
     PAGES = {xiv+551},
      ISBN = {978-1-4614-8023-5; 978-1-4614-8024-2},
   MRCLASS = {22-01 (22C05 22E46 22E60)},
  MRNUMBER = {3136522},
MRREVIEWER = {Ben\ Lewis\ Cox},
       URL = {https://doi.org/10.1007/978-1-4614-8024-2},
}

@article {GelfandGraev,
    AUTHOR = {Gelfand, Izrail Moiseevich and Graev, Mark Iosifovich},
     TITLE = {Unitary representations of the real unimodular group
              (principal nondegenerate series)},
   JOURNAL = {Amer. Math. Soc. Transl. (2)},
  FJOURNAL = {Amer. Math. Soc. Transl. (2)},
    VOLUME = {2},
      YEAR = {1956},
     PAGES = {147--205},
   MRCLASS = {17.0X},
  MRNUMBER = {76291},
}

@article {Graev,
    AUTHOR = {Graev, Mark Iosifovich},
     TITLE = {Unitary representations of real simple {L}ie groups},
   JOURNAL = {Amer. Math. Soc. Transl. (2)},
  FJOURNAL = {Amer. Math. Soc. Transl. (2)},
    VOLUME = {16},
      YEAR = {1960},
     PAGES = {393--396},
   MRCLASS = {46.00 (22.00)},
  MRNUMBER = {117598},
       URL = {https://doi.org/10.1090/trans2/016/18},
}

@article {Matsuki,
    AUTHOR = {Matsuki, Toshihiko},
     TITLE = {The orbits of affine symmetric spaces under the action of
              minimal parabolic subgroups},
   JOURNAL = {J. Math. Soc. Japan},
  FJOURNAL = {Journal of the Mathematical Society of Japan},
    VOLUME = {31},
      YEAR = {1979},
    NUMBER = {2},
     PAGES = {331--357},
      ISSN = {0025-5645,1881-1167},
   MRCLASS = {53C35 (22E46 32M15)},
  MRNUMBER = {527548},
MRREVIEWER = {S.\ Murakami},
       URL = {https://doi.org/10.2969/jmsj/03120331},
}

@article {Aomoto,
    AUTHOR = {Aomoto, Kazuhiko},
     TITLE = {On some double coset decompositions of complex semisimple
              {L}ie groups},
   JOURNAL = {J. Math. Soc. Japan},
  FJOURNAL = {Journal of the Mathematical Society of Japan},
    VOLUME = {18},
      YEAR = {1966},
     PAGES = {1--44},
      ISSN = {0025-5645,1881-1167},
   MRCLASS = {22.50 (32.32)},
  MRNUMBER = {191994},
MRREVIEWER = {W.\ T.\ van Est},
       URL = {https://doi.org/10.2969/jmsj/01810001},
}

@article {Wolf,
    AUTHOR = {Wolf, Joseph A.},
     TITLE = {The action of a real semisimple group on a complex flag
              manifold. {I}. {O}rbit structure and holomorphic arc
              components},
   JOURNAL = {Bull. Amer. Math. Soc.},
  FJOURNAL = {Bulletin of the American Mathematical Society},
    VOLUME = {75},
      YEAR = {1969},
     PAGES = {1121--1237},
      ISSN = {0002-9904},
   MRCLASS = {32.32 (22.00)},
  MRNUMBER = {251246},
MRREVIEWER = {W.\ Klingenberg},
       URL = {https://doi.org/10.1090/S0002-9904-1969-12359-1},
}

@book{onishchik2012lie,
  title={Lie groups and algebraic groups},
  author={Onishchik, Arkadij L. and Vinberg, Ernest B.},
  year={2012},
  publisher={Springer Science $\&$ Business Media}
}

@incollection{RS93,
    AUTHOR = {Richardson, Roger Wolcott and Springer, Tonny Albert},
     TITLE = {Combinatorics and geometry of {$K$}-orbits on the flag
              manifold},
 BOOKTITLE = {Linear algebraic groups and their representations ({L}os
              {A}ngeles, {CA}, 1992)},
    SERIES = {Contemp. Math.},
    VOLUME = {153},
     PAGES = {109--142},
 PUBLISHER = {Amer. Math. Soc., Providence, RI},
      YEAR = {1993},
      ISBN = {0-8218-5161-6},
   MRCLASS = {14L30 (20G15)},
  MRNUMBER = {1247501},
       URL = {https://doi.org/10.1090/conm/153/01309},
}

@article {RS94,
    AUTHOR = {Richardson, Roger Wolcott and Springer, Tonny Albert},
     TITLE = {Complements to: ``{T}he {B}ruhat order on symmetric
              varieties'' [{G}eom. {D}edicata {\bf 35} (1990), no. 1-3,
              389--436; {MR}1066573 (92e:20032)]},
   JOURNAL = {Geom. Dedicata},
  FJOURNAL = {Geometriae Dedicata},
    VOLUME = {49},
      YEAR = {1994},
    NUMBER = {2},
     PAGES = {231--238},
      ISSN = {0046-5755,1572-9168},
   MRCLASS = {20G15 (20G20 22E60)},
  MRNUMBER = {1266276},
MRREVIEWER = {Aloysius\ Helminck},
       URL = {https://doi.org/10.1007/BF01610623},
}

@article {Fulton92,
    AUTHOR = {Fulton, William},
     TITLE = {Flags, {S}chubert polynomials, degeneracy loci, and
              determinantal formulas},
   JOURNAL = {Duke Math. J.},
  FJOURNAL = {Duke Mathematical Journal},
    VOLUME = {65},
      YEAR = {1992},
    NUMBER = {3},
     PAGES = {381--420},
      ISSN = {0012-7094,1547-7398},
   MRCLASS = {14C17 (14M12 14M15)},
  MRNUMBER = {1154177},
MRREVIEWER = {Piotr\ Pragacz},
       URL = {https://doi.org/10.1215/S0012-7094-92-06516-1},
}

@article {Fulton96a,
    AUTHOR = {Fulton, William},
     TITLE = {Determinantal formulas for orthogonal and symplectic
              degeneracy loci},
   JOURNAL = {J. Differential Geom.},
  FJOURNAL = {Journal of Differential Geometry},
    VOLUME = {43},
      YEAR = {1996},
    NUMBER = {2},
     PAGES = {276--290},
      ISSN = {0022-040X,1945-743X},
   MRCLASS = {14C17 (14M12 14M15)},
  MRNUMBER = {1424427},
MRREVIEWER = {E.\ Aky\i ld\i z},
       URL = {http://projecteuclid.org/euclid.jdg/1214458108},
}

@inproceedings {Fulton96b,
    AUTHOR = {Fulton, William},
     TITLE = {Schubert varieties in flag bundles for the classical groups},
 BOOKTITLE = {Proceedings of the {H}irzebruch 65 {C}onference on {A}lgebraic
              {G}eometry ({R}amat {G}an, 1993)},
    SERIES = {Israel Math. Conf. Proc.},
    VOLUME = {9},
     PAGES = {241--262},
 PUBLISHER = {Bar-Ilan Univ., Ramat Gan},
      YEAR = {1996},
   MRCLASS = {14M15 (14C17 14M12)},
  MRNUMBER = {1360506},
}

@article {Graham,
    AUTHOR = {Graham, William},
     TITLE = {The class of the diagonal in flag bundles},
   JOURNAL = {J. Differential Geom.},
  FJOURNAL = {Journal of Differential Geometry},
    VOLUME = {45},
      YEAR = {1997},
    NUMBER = {3},
     PAGES = {471--487},
      ISSN = {0022-040X,1945-743X},
   MRCLASS = {14M15 (17B99)},
  MRNUMBER = {1472885},
MRREVIEWER = {Eric\ N.\ Sommers},
       URL = {http://projecteuclid.org/euclid.jdg/1214459839},
}

@article{mcgovern2009pattern,
  title={Pattern avoidance and smoothness of closures for orbits of a symmetric subgroup in the flag variety},
  author={McGovern, William M. and Trapa, Peter E.},
  journal={Journal of Algebra},
  volume={322},
  number={8},
  pages={2713--2730},
  year={2009},
  publisher={Elsevier}
}

@article{mcgovern2009closures,
  title={Closures of K-orbits in the flag variety for $U(p, q)$},
  author={McGovern, William M.},
  journal={Journal of Algebra},
  volume={322},
  number={8},
  pages={2709--2712},
  year={2009},
  publisher={Elsevier}
}

@article{wyser2014polynomials,
  title={Polynomials for $\GL_p \times \GL_q$ orbit closures in the flag variety},
  author={Wyser, Benjamin J. and Yong, Alexander},
  journal={Selecta Mathematica},
  volume={20},
  pages={1083--1110},
  year={2014},
  publisher={Springer}
}

@article{colarusso2014gelfand,
  title={The Gelfand--Zeitlin integrable system and K-orbits on the flag variety},
  author={Colarusso, Mark and Evens, Sam},
  journal={Symmetry: Representation Theory and Its Applications: In Honor of Nolan R. Wallach},
  pages={85--119},
  year={2014},
  publisher={Springer}
}

@article{lam2021back,
  title={Back stable Schubert calculus},
  author={Lam, Thomas and Lee, Seung Jin and Shimozono, Mark},
  journal={Compositio Mathematica},
  volume={157},
  number={5},
  pages={883--962},
  year={2021},
  publisher={London Mathematical Society}
}

@article{lee2019combinatorial,
  title={Combinatorial description of the cohomology of the affine flag variety},
  author={Lee, Seung Jin},
  journal={Transactions of the American Mathematical Society},
  volume={371},
  number={6},
  pages={4029--4057},
  year={2019}
}

@article{lusztig1983singularities,
  title={Singularities of closures of K-orbits on flag manifolds},
  author={Lusztig, George and Vogan Jr, David A.},
  journal={Inventiones mathematicae},
  volume={71},
  number={2},
  pages={365--379},
  year={1983},
  publisher={Springer}
}

@book{hall2013lie,
  title={Lie groups, Lie algebras, and representations},
  author={Hall, Brian C.},
  year={2013},
  publisher={Springer}
}

@book{horn2012matrix,
  title={Matrix analysis},
  author={Horn, Roger A. and Johnson, Charles R.},
  year={2012},
  publisher={Cambridge university press}
}

@article{wyser2017polynomials,
  title={Polynomials for symmetric orbit closures in the flag variety},
  author={Wyser, Benjamin J. and Yong, Alexander},
  journal={Transformation Groups},
  volume={22},
  pages={267--290},
  year={2017},
  publisher={Springer}
}
